\documentclass[a4paper,10pt]{amsart}
\usepackage{amssymb,amsmath,amsfonts,amsthm}
\usepackage{graphicx}
\graphicspath{ {./images/} }
 
\usepackage{psfrag}
\usepackage{mathtools}
\usepackage{color}
\usepackage{enumitem}

\theoremstyle{plain}
\newtheorem{main}{Theorem}

\newtheorem{maincor}[main]{Corollary}
\newtheorem{theorem}{Theorem}[section]
\newtheorem{lemma}[theorem]{Lemma}
\newtheorem{proposition}[theorem]{Proposition}
\newtheorem{corollary}[theorem]{Corollary}
\theoremstyle{remark}
\newtheorem{remark}[theorem]{Remark}
\newtheorem{definition}{Definition}

\newcommand\numberthis{\addtocounter{equation}{1}\tag{\theequation}}

\newcommand{\Leb}{\operatorname{vol}}

\newcommand{\C}{\operatorname{C}}
\newcommand{\Sing}{\operatorname{Sing}}

\newcommand{\Card}{\operatorname{Card}}
\newcommand{\Jac}{\operatorname{Jac}}
\newcommand{\diam}{\operatorname{diam}}

           \def\ea{\end{array}}
          \def\ec{\end{center}}
     \def\ed{\end{description}}
        \def\ee{\end{equation}}
       \def\eea{\end{eqnarray}}
     \def\eeaa{\end{eqnarray*}}
 \def\et{\end{thebibliography}}

\def\Orb{{\rm Orb}}
\def\Diff{{\rm Diff}}

\def\Sing{{\rm Sing}}

\def\supp{\operatorname{supp}}

\def\cA{{\mathcal A}}
\def\cD{{\mathcal D}}

\def\cU{{\mathcal U}}

\def\cR{{\mathcal R}}

\def\cF{{\mathcal F}}

\def\cN{{\mathcal N}}
\def\cP{{\mathcal P}}

\def\cR{{\mathcal R}}

\def\length{\operatorname{length}}

\def\RR{{\mathbb R}}

\def\ZZ{{\mathbb Z}}
\def\NN{{\mathbb N}}

\def\tl{\tilde{\Lambda}}
\def\bcf{\overline{\cF}}

\title[Lorenz-like flows]{Entropy theory for sectional hyperbolic flows}

\author{Maria Jos\' e Pacifico, Fan Yang and Jiagang Yang}
\date{\today}
\thanks{M. J. P. and J.Y. are partially supported by CNPq, FAPERJ, PROEX-CAPES. F.Y. would like to thank the hospitality of Southern University of Science and Technology of China (SUSTC), where part of this work is done.}

\address{Instituto de Matem\'atica, Universidade Federal do Rio de Janeiro, C. P. 68.530, CEP 21.945-970,  Rio de Janeiro, RJ, Brazil.}
 \email{pacifico@im.ufrj.br }
\address{Department of Mathematics, University of Oklahoma, Norman, Oklahoma, USA.}
\email{fan.yang-2@ou.edu}
\address{Department of Mathematics, Southern University of Science and Technology of China, Guangdong, China; and 
Departamento de Geometria, Instituto de Matem\'atica e Estat\'istica, Universidade
Federal Fluminense, Niter\'oi, Brazil.}
\email{yangjg\@@impa.br}
\begin{document}

\begin{abstract}
We use entropy theory as a new tool to study sectional hyperbolic flows in any dimension. We show that for $C^1$ flows,  every sectional hyperbolic set $\Lambda$ is entropy expansive, and the topological entropy varies continuously with the flow. Furthermore, if $\Lambda$ is Lyapunov stable, then it has positive entropy; in addition, if $\Lambda$ is a chain recurrent class, then it contains a periodic orbit. As a corollary, we prove that for $C^1$ generic flows, every Lorenz-like class is an attractor. \end{abstract}

\maketitle

\tableofcontents

\section{Introduction}
About half a century ago, Lorenz published his famous article~\cite{Lo63} in which he used computer-aided numerical simulation to study the following system, which is now known as the Lorenz equations:
\begin{equation}\label{e.Lorenz}
\begin{cases} 
    \dot{x}=-\sigma x+\sigma y\quad\quad\quad\sigma=10\\
    \dot{y}=ry-xz\quad\quad\quad\quad r=28\\
    \dot{z}=-bz+xy\quad\quad\quad\ b=8/3.
\end{cases}
\end{equation}
Numerical simulations  for an open neighborhood of the chosen parameters suggested that almost all points in the phase space tend to a {\em{chaotic attractor}} .

An attractor is a bounded region in the phase space, invariant under time evolution, to which the forward trajectories of most (positive probability) or, even, all nearby points converge. What makes an attractor chaotic is the fact that trajectories converging to the attractor are sensitive with respect to initial data: trajectories of any two nearby points eventually diverge under time evolution.

Lorenz equations prove to be very resistant to rigorous mathematical analysis, from both conceptual (existence of the equilibrium accumulated by regular orbits prevents the attractor to be hyperbolic) as well numerical (solutions slow down as they pass near the equilibrium, which means unbounded return times and, thus, unbounded integration errors) point of view. Based on numerical experiments, Lorenz conjectured that the flow
generated by the equations (\ref{e.Lorenz}) presents a volume zero chaotic attractor  which is robust (it persists under small perturbation of the parameters). This attractor is called a {\em{Lorenz attractor}}; it
 has a  butterfly shape and displays  extremely rich dynamical properties.
It is robust in the sense that every nearby flow also possesses  an attractor with similar properties. Part of the reason for the richness of the Lorenz attractor is the fact that it has an {\em equilibrium or singularity}, i.e., a  point where the vector field vanishes, that is accumulated by regular orbits (orbits through points where the corresponding vector field does not vanish) which prevents the flow from being uniformly hyperbolic.

In the seventies, a geometric Lorenz model for this attractor was proposed in ~\cite{G76, ABS, GW79}. These models are flows in three dimension for which one can rigorously prove the existence of a chaotic attractor that contains an equilibrium point of the flow, which is an accumulation point of typical regular solutions.

Finally, the above stated existence of a chaotic attractor for the original Lorenz system was not proved until the year 2000, when Tucker did so with a computer-aided proof ~\cite{T99, T02}.

In order to describe the hyperbolicity of a Lorenz  flow, or more generally, invariant sets for a three-dimensional flow that contains equilibria, Morales, Pacifico and Pujals~\cite{MPP99} proposed the notion of {\em singular hyperbolicity}, which requires the flow to have a one-dimensional uniformly contracting direction, and a two-dimensional sub-bundle containing the flow direction of regular points, on which the flow is volume expanding. 
It is shown in~\cite{MPP04} that every robust attractor of a three-dimensional flow must be singular hyperbolic. Later, Ara\'ujo {\em et al} proved in~\cite{APPV} that every singular hyperbolic attractor for $3$-flows is expansive. The key technique in their proof is the linearization near singularities, and being an attractor allows them to take a collection of cross sections, thus reduce the flow to a two-dimensional return map. 

The notion of singular hyperbolicity was later generalized to {sectional hyperbolicity} for high dimensional flows, see~\cite{LGW, MM}. 
Roughly speaking, a compact invariant set $\Lambda$ for a flow $\phi_t$  is sectional hyperbolic if the tangent bundle at $\Lambda$ splits into two complementary directions $E^ s \oplus F^{cu}$ with  $E^ s$ uniformly contracting, dominates $F^{cu}$  and for any subspace $V \subset F^{cu}$ with dimension greater or equal two, the derivative  $D\phi_t$ restricted to $V$ is volume expanding, see Definition \ref{d.sectionalhyperbolic} below.
Initial results in the theory of sectional-hyperbolic flows were obtained in \cite{M07}. Nowadays there already exists a relatively good advance in this theory and we recommend the interested reader to~\cite{BM15} and the references therein.

Despite all these progress, it is worth mentioning that
up to now, the classical method to deal with flows presenting equilibria accumulated in a robust way by regular orbits, still highly relies on the construction of  cross sections and on the analysis of the 
associated Poincar\'e map, thus limiting its usage to dimension three.
In this paper, we introduce a new method to study sectional hyperbolic flows, in particular to study flows presenting equilibria accumulated by regular orbits, which is based on a new development of entropy theory in~\cite{LVY}. This new method
provides a breakthrough on the analysis of such class of flows on any dimension, giving a fruitful improvement in the actual knowledge and suggests   a new method: using entropy theory,  to study flows in general.

For simplicity, we will assume for now that all the singularities are hyperbolic. However, this assumption is not essential and can be removed with only a slight modification to the proof. We will deal with the case of  non-hyperbolic singularities in the appendix. See Theorems~\ref{m.h-expansive.non-hyperbolic} and~\ref{m.positiveentropy.non-hyperbolic} in the Appendix.
Furthermore, our work shows that the presence of equilibria, unlike in the topological theory for flows where they form an essential obstruction, turns out to be quite beneficial in the entropy theory.  This is reflected by the proof of Theorem~\ref{m.h-expansive}.

In order to announce our results in a precise way, let us introduce some definitions needed to comprehend the text.

\begin{definition}\label{d.sectionalhyperbolic}
A compact invariant set $\Lambda$ of a flow $X$ is called {\em sectional hyperbolic}, if it admits a dominated splitting $E^s\oplus F^{cu}$, such that $E^s$ is uniformly contracting, and $F^{cu}$ is {\em sectional-expanding:}  there are constants $C,\lambda>0$ such that for every $x\in\Lambda$ and any subspace $V_x\subset F^{cu}_x$ with $\dim V_x \ge 2$, we have 
$$
|\det D\phi_t(x)|_{V_x}| \ge C e^{\lambda t} \mbox{ for all } t>0.
$$
We  call $\lambda$ the {\em sectional volume expanding rate on $F^{cu}$}. 
\end{definition}

\begin{definition}
For $\varepsilon>0, T>0$, a finite sequence $\{x_i\}_{i=0}^n$ is called an {\em $(\varepsilon, T)$-chain} if there exists $\{t_i\}_{i=0}^{n-1}$ such that $t_i>T$ and $d(\phi_{t_i}(x_i),x_{i+1})<\varepsilon$  for all $i=0,\ldots, n-1$. We say that $y$ is {\em chain attainable from $x$}, if there exists $T>0$ such that for all $\varepsilon>0$, there exists an $(\varepsilon, T)$-chain $\{x_i\}_{i=0}^n$ with $x_0=x$ and $x_n=y$. It is straightforward to check that chain attainability is an equivalent relation on the set ${\rm CR}(X) = \overline{\{x: \forall t>0, \mbox{ there is an } (\varepsilon, T)\mbox{-chain with } x_0 = x_n = x\mbox{ and }\sum_i t_i > t\}}$. Each equivalent class under this relation is then called a {\em chain recurrent class}. A chain recurrent class $C$ is  {\em non-trivial} if it is not a singularity or a periodic orbit.
\end{definition}

\begin{definition}\label{d.lorenzclass}
A compact invariant set $\Lambda$ is a {\em Lorenz-like class}, if it contains both singularities and regular points, and satisfies the following conditions:
\begin{enumerate}
\item $\Lambda$ is a chain recurrent class.
\item $\Lambda$ is {\em Lyapunov stable}, i.e., there is a sequence of compact neighborhoods
$\{U_i\}$ such that:
\begin{itemize}
\item $U_1 \supset U_2\supset\cdots$ and $\bigcap_iU_i = \Lambda$;
\item for each $i \ge 1$, $\phi_t(U_{i+1}) \subset U_i$ for any $t \ge 0$.
\end{itemize}
\item $\Lambda$ is  sectional-hyperbolic.
\end{enumerate}
\end{definition}
Recall that a $C^1$ flow is a {\em star flow} if for any nearby flow, its critical elements,
i.e., singularities and periodic orbits, are all hyperbolic.
Recently, it is proven in~\cite{GSW} that for generic star flows, every non-trivial Lyapunov stable chain recurrent class is Lorenz-like.

Note that in the definition of a Lorenz-like class we \textbf{neither} assume the singularities to be hyperbolic \textbf{nor} the class to be isolated. As a result, tools developed for three-dimensional singular hyperbolic attractors (see~\cite{ArPa10}) become invalid for higher dimensional sectional hyperbolic flows: non-isolated invariant set makes it difficult to take cross sections, and non-hyperbolic singularities make linearization impossible. Also recall that even when the singularities are hyperbolic, one still need extra assumptions on the eigenvalues of the singularity and sufficient smoothness in order to use linearization. To overcome these difficulties, we will look at the time-one map of the flow, and use the fake foliations developed in~\cite{LVY} to studied the expanding property in a neighborhood of the singularities.

On the other hand, entropy theory for discrete and continuous time systems (without singularities) have been developed for over 50 years and is proven to be quite successful. The {\em entropy expansiveness}, first introduced by Bowen in~\cite{B72}, is one of the key reasons that hyperbolic systems have nice properties in regard to entropy perspective. In particular, Bowen proved that if the system is robust entropy expansive, then the metric entropy is upper semi-continuous at the invariant measure, and the topological entropy varies upper semi-continuously with the system. As an important corollary, there always exists an equilibrium state (a measure whose metric pressure equals the topological pressure of the system) for every  continuous potential. In particular, there exists a measure of maximal entropy.

\begin{definition}\label{d.nearlyrobust}
Let $\Lambda$  be a  compact invariant set  for a $C^ 1$ flow $\phi_t$ and $\delta > 0$. Let $f=\phi_1$ be the time-one map of $X$.
We say that $\phi_t$  is $h$-expansive or entropy expansive on $\Lambda$ at scale $\delta$, if 
the set $ {B_\infty}(x,\delta) = \{ y:d({f^n}(x),{f^n}(y)) < \delta \mbox{ for all } n \in\ZZ\}$ (below we call it the $(\infty,\delta)$-Bowen ball at $x$) has zero topological entropy for each $ x \in \Lambda$. 
The map $\phi_t$ is robustly $\delta$-entropy expansive near $\Lambda$  if there is  a neighborhood $\cU$ of $\phi_t$ in the $C^1$ topology and a neighborhood $U$ of $\Lambda$, such that for every $\psi_t \in \cU$, the maximal invariant set of $\psi_t $ in $U$ is entropy expansive at scale $\delta$.
\end{definition}

\begin{main}\label{m.h-expansive}
Let $\Lambda$ be a compact invariant set that is sectional hyperbolic for a $C^1$ flow $\phi_t$, with all the singularities in $\Lambda$ hyperbolic. Then $\phi_t$ is robustly $\delta$-entropy expansive near $\Lambda$.
\end{main}

We also obtain continuity of the topological entropy for sectional hyperbolic invariant sets:

\begin{main}\label{m.continuity}
Let $\Lambda$ be a sectional hyperbolic compact invariant set for a $C^1$ flow $\phi_t$, with all singularities in $\Lambda$ hyperbolic. Then there is a neighborhood $U$ of $\Lambda$, such that $h_{top}(\cdot|_{\tilde{\Lambda}})$ is continuous at $\phi_t$, where $\tilde{\Lambda}$ is the maximal invariant set in $U$. More precisely, let $\phi^n_t$ be a sequence of $C^1$ flows with $\phi^n_t \xrightarrow{n\to\infty} \phi_t$  in $C^1$ topology, and denote by $\tilde{\Lambda}_n$ the maximal invariant set of $\phi^n_t$ in $U$, then
$$
\lim_{n\to\infty}h_{top}(\phi^n_t|_{\tilde{\Lambda}_n}) = h_{top}(\phi_t|_{\tilde{\Lambda}}).
$$
Furthermore, if $h_{top}(\phi_t|_\Lambda)>0$ then there are periodic orbits arbitrarily close to $\Lambda$. When $\Lambda$ is a chain recurrent class, such periodic orbits are indeed contained in $\Lambda$.
\end{main}

Note that in Theorems~\ref{m.h-expansive} and~\ref{m.continuity}, we do not need $\Lambda$ to be Lyapunov stable. If $\Lambda$ is Lyapunov stable, then we get positive topological entropy:

\begin{main}\label{m.positiveentropy}
Let $\Lambda$ be a compact invariant set that is sectional hyperbolic for a $C^1$ flow $\phi_t$, with all the singularities in $\Lambda$ hyperbolic. Furthermore, assume that $\Lambda$ is  Lyapunov stable. Then  $h_{top}(\phi_t|_\Lambda)\ge \lambda>0$, where $\lambda$ is the sectional volume expanding rate on $F^{cu}$.
\end{main}

We apply the previous theorem to Lorenz-like classes and obtain the following corollary:

\begin{maincor}\label{mc.robusthexpansive}
Let $\Lambda$ be a Lorenz-like class for a $C^1$ flow $\phi_t$ with all singularities hyperbolic.  Then  $\Lambda$ is robustly $\delta$-entropy expansive, has positive topological entropy, and contains a periodic orbit.
\end{maincor}

Recall that a property is  $C^1$ generic if it holds on a residual set under $C^1$ topology. With the help of the periodic orbit obtained in Corollary~\ref{mc.robusthexpansive}, we show that Lorenz-like classes are indeed attractors.

\begin{maincor}\label{mc.attractor}
$C^1$ generically, every Lorenz-like class is an attractor and contains a periodic orbit.
\end{maincor}

Now let us  explain how the entropy theory is used in the proof.
In Section~\ref{ss.expansive}, we study the time-one map $f=\phi_1$ of the flow  in a neighborhood of  the sectional hyperbolic set $\Lambda$, which has a  dominated splitting $E^s \oplus F^{cu}$. This enables us to  use  `fake foliations', which are invariant under the map $f$, but are generally not preserved by the flow. We will prove that the infinite Bowen balls are contained in those fake-foliations (Lemma~\ref{l.bowenball}), and the flow saturates the fake foliation for points in the infinite Bowen balls (Corollary~\ref{c.saturation}). Using the fake foliations, we establish a local product structure near singularities (Lemma~\ref{l.expansion.near}), which allows us to use the center foliation near singularities and  establish some expanding property near neighborhoods of singularities {\em{without }} using linearization. Using those properties, we will prove Theorem~\ref{m.almostexpansive} (see the beginning of Section~\ref{ss.expansive}) which is a stronger version of Theorem~\ref{m.h-expansive}: the set of points where the infinite Bowen balls are degenerate has full measure under every ergodic, invariant probability measure.

Note that our proof for the entropy expansiveness relies heavily on the sectional hyperbolic splitting. On the other hand, the example of Bonatti and da Luz in~\cite{BD}  does not admit a sectional hyperbolic splitting. It will be a challenging problem to obtain the entropy expansiveness for such classes. 

In Section~\ref{ss.positiveentropy}, we prove Theorem~\ref{m.positiveentropy} by showing that the time-one map $f$ on a neighborhood of $\Lambda$ have positive topological entropy, using the volume expansion rate on the $F^{cu}$ bundle as a lower bound. Then Theorem~\ref{m.positiveentropy} will follow by taking a sequence of such neighborhoods shrinking to $\Lambda$ and using the upper semi-continuity of the metric entropy. 

In Section~\ref{s.continuity}, we prove Theorem~\ref{m.continuity} using a similar argument as Katok~\cite{K80} and a shadowing lemma of Liao~\cite{Liao}, which allows the pseudo orbit to pass near singularities. The proof of Corollary~\ref{mc.robusthexpansive} and~\ref{mc.attractor} 
can be found at the end of Section~\ref{s.continuity}.

Finally, in the Appendix we will revisit the proof of Theorem~\ref{m.h-expansive}, without assuming that singularities are hyperbolic.

\section{Preliminaries}\label{s.preliminary}
Throughout this paper, $X$ will be a vector field that is  $C^1$ on a $d$-dimensional compact manifold $M$. Denote by $\Sing(X)$ (sometimes we also write $\Sing(\phi_t)$) the set of singularities of $X$, $\phi_t$ the flow generated by $X$, and $f = \phi_1$ the time-one map of $\phi_t$. We will write $\Phi_t$ for  the tangent flow, i.e., $\Phi_t  = D\phi_t: TM \to TM$.

\subsection{Dominated splitting} Let $g\in\Diff^1(M)$ be a diffeomorphism on $M$. We say that $g$ has a {\em dominated splitting $E\oplus F$}, if $TM$ can be decomposed into continuous, $Dg$ invariant subbundles $E$ and $F$, such that for some $L>0$, we have 
$$
\frac{\|Dg^L_x(u)\|}{\|u\|} \le \frac12 \frac{\|Dg^L_x(v)\|}{\|v\|}
$$
for every $x \in M$ and every  non-zero vectors $u \in E(x), v \in F(x)$. The dominated splitting on an invariant set $\Lambda$ can be defined in a similar way, with $TM$ replaced by $T_\Lambda M$.

For $a > 0$ and $x \in M$, a {\em $(a, F)$-cone} on the tangent space $T_xM$ is defined as
$$
C_a(F_x) = \{v:  v = v_E + v_F \mbox{ where } v_E \in E, v_F\in F \mbox{ and } \|v_E\|< a\|v_F\|\}\cup\{0\}.
$$
When $a$ is sufficiently small, the cone field $C_a(F_x)$, $x \in M$, is forward invariant by $Dg$,
i.e., there is $\lambda < 1$ such that for any $x \in  M$, $Dg_x(C_a(F_x)) \subset  C_{\lambda a}(F_{g(x)})$. Similarly, we can define the $(a, E)$-cone $C_a(E_x)$, which is backward invariant by $Dg$. When no
confusing is caused, we call the two families of cones by $F$ cones and $E$ cones.

The images of the cones under the exponential map are also forward or backward
invariant.  To be more precise, fix $\varepsilon_0 > 0$ small enough, such that the exponential
map is well-defined on the $\varepsilon_0$-ball in the tangent space. We denote by $C^M_a(F_x)$ the image of $C_a(F_x)$ under the exponential map restricted to the set $B_{\varepsilon_0}(0) \cap C_a(F_x) \subset  T_xM$ and call $C^M_a(F_x)$ a {\em local $F$ cone in $B_{\varepsilon_0}(x)$}. Then for any $x \in M$, we have:
$$
g\left(C^M_a(F_x) \cap B_{\varepsilon_0/\|g\|_{C^1}}(x)\right)\subset C^M_{\lambda a}(F_{f(x)}).
$$
In the same way we can define $C^M_a(E_x)$.

\begin{definition}
Let $D$ be a $C^1$ disk with dimension equals to $\dim F$. We say $D$ is:
\begin{itemize}
\item tangent to the $F$ cone if for any $x \in D$, $T_xD \subset  C_a(F_x)$;
\item tangent to the {\bf local $F$ cone at $x$ } if $D \subset C^M_a(F_x)$;
\item tangent to the {\bf local $F$ cone} if for any $y \in D$, we have $D \subset C^M_a(F_y)$.
\end{itemize}
\end{definition} 
$D$ is tangent to the local $F$ cone implies that it is tangent to the $F$ cone. Conversely, if
$D$ is tangent to the $F$ cone, then it can be divided into finitely many sub-disks, each of which is tangent to the local $F$ cone.
\begin{remark}
Topologically, for $a$ small enough, the local cones $C^M_a(E_x)$ and $C^M_a(F_x)$ are transverse to each other, that is, $C^M_a(E_x) \cap C^M_a(F_x) = \{x\}$.
\end{remark}
\begin{remark}
Suppose $D$ is a disk with dimension $\dim F$ such that $TD$ is transverse to $E$
bundle, then there is $n > 0$ sufficiently large, such that $g^n(D)$ is tangent to the $F$ cone.
Hence, it can be divided into finitely many connected pieces:
$
g^n(D) = \bigcup_{i=1}^l D_i
$,
such that each piece $D_i$ is tangent to the  local $F$ cone.
\end{remark}

The proof of the next lemma is simple and thus omitted.
\begin{lemma}\label{l.boundedvolume}
There is a constant $K>0$ such that for every $x \in M$ and any disk $D\subset B_{\varepsilon_0}(x)$ tangent to a local $F$ cone, we have $\Leb(D)<K.$
\end{lemma}

By the forward invariance of the $F$ cone field, we obtain:
\begin{lemma}\label{l.invariantcone}
For every $x\in M$, $\varepsilon < \frac{\varepsilon_0}{\|g\|_{C^1}}$ and $n>0$, if  $D \subset B_\varepsilon(x)$ is tangent to a local $F$ cone at $x$, with  $g^i(D) \subset B_\varepsilon(g^i(x))$ for every $0 \le i \le n$, then $g^n(D)$ is tangent to the local $F$ cone at $g^n(x)$.
\end{lemma}

One can easily check that if $g$ has dominated splitting on an invariant set $\Lambda$ instead of the entire manifold $M$, then the invariant cone fields $C_a(E)$ and $C_a(F)$ can be extended to a neighborhood of $\Lambda$. See~\cite[Appendix B]{BDV}. One can then define the local cones $\{C_a^M\}$ in the same way, and the above lemmas hold for points in the neighborhood of $\Lambda. $

For $\varepsilon > 0$ and $n \ge 1$,
we consider the dynamical ball of radius $\varepsilon > 0$ and length $n$ around $x \in M$, defined as $B_n(x,\varepsilon) = \{y: d(g^k(x), g^k(y)) < \varepsilon, k = 0,\ldots, n-1\}$.  This is also called the $(n,\varepsilon)$-Bowen ball and plays an important role in the study of topological entropy.  As a simple corollary of Lemma~\ref{l.invariantcone}, we have:

\begin{lemma}\label{l.boundedvolumebowen}
Suppose $D$ is a disk with dimension $\dim F$ and tangent to the local $F$ cone. Then for any  $x \in  D$ and $\varepsilon < \varepsilon_0/\|g\|_{C^1}$, one has
$$ 
\Leb(g^n(D \cap B_n(x, \varepsilon))) \le  K.
$$
\end{lemma}
\begin{proof}
This Lemma follows easily from Lemma~\ref{l.invariantcone} and the observation that\\ $g^n(B_n(x,\varepsilon))\subset B_\varepsilon(g^n(x))$.
\end{proof}

\subsection{Entropy of continuous maps}
In this section  $g : M \to M$ will be a continuous map and $K$  a subset of $M$ not necessarily invariant. 
A set $E \subset M$ is  $(n, \varepsilon)$-spanning for $K$ if for any $x \in  K$, there is $y \in E$ such
that $d(g^i(x), g^i(y)) < \varepsilon$ for all $0 \le i \le n-1$. In other words, the dynamical balls $B_n(y, \varepsilon)$, $y \in E$ cover $K$. Let $r_n(K, \varepsilon)$ denote the smallest cardinality of any $(n, \varepsilon)$-spanning set of $K$, and
$$
r(K, \varepsilon) = \limsup_{n\to\infty}\frac1n\log r_n(K, \varepsilon).
$$
The topological entropy of $g$ on $K$ is then defined as 
$$
h_{top}(g, K) = \lim_{\varepsilon\to 0}r(K, \varepsilon),
$$
and the topological entropy of $g$ is defined as $h_{top}(g) = h_{top}(g, M)$. 
\vspace{0.1cm}

For each $x \in M$ and $\varepsilon > 0$, let $B_\infty(x, \varepsilon) = \{y : d(g^n(x), g^n(y)) < \varepsilon \mbox{ for } n \in \ZZ\}$ be the two-sided $(\infty,\varepsilon)$-Bowen ball at $x$. The map $g$ is  {\em $\varepsilon$-entropy expansive} if
$$
\sup_{x\in M}h_{top}(g, B_\infty(x, \varepsilon)) = 0.
$$
In other words, the $(\infty,\varepsilon)$-Bowen ball has zero entropy for all $x$.
It is well known (see for example~\cite{B72}) that if $g$ is $\varepsilon$-entropy expansive, then the topological entropy ``stabilizes'' at $\varepsilon$, that is, $h_{top}(g) = r(M,\varepsilon)$.

Next we consider the {\em{metric entropy}} of an invariant measure. Let $\mu$ be an invariant measure and $\cA$ a finite measurable partition. The metric entropy of $\mu$ corresponding to the partition $\cA$ is defined as
$$
h_\mu(\cA) = -\lim_{n\to\infty}\frac{1}{n} \sum_{B\in\cA^{n-1}_0}\mu(B) \log \mu(B),
$$
where $\cA^{n-1}_0$ is the $(n-1)$th joint of $\cA$:
$$
\cA^{n-1}_0 = \cA \vee g^{-1}\cA \vee \cdots \vee g^{-(n-1)}\cA.
$$
The metric entropy of an invariant measure $\mu$ is defined as
$$
h_{\mu} = \sup_{ \cA \mbox{ is a finite partition}}\{h_{\mu}(\cA)\}.
$$
By the variational principle, $h_{top}(g) = \sup_{\mu\in M_{inv}(g)} h_{\mu}$, where $M_{inv}(g)$ denotes
the space of invariant probabilities of $g$. If $g$ is $\varepsilon$-entropy expansive, then for every finite partition $\cA$ with $\diam \cA<\varepsilon$, we have 
$h_{\mu} = h_\mu(\cA).$ See~\cite[Theorem 3.5]{B72}.

In general, for maps with finite differentiability, metric entropy is not necessarily  upper semi-continuous with respect to the invariant measures, and the topological entropy may not be achieved by the metric entropy of any invariant measure, see for example~\cite{DN,N}. However, it is proven in~\cite{B72} that if $g$ is $\varepsilon$-entropy expansive (or asymptotically $h$-expansive, meaning that $\lim_{\varepsilon\to 0}\sup_{x\in M}h_{top}(g, B_\infty(x, \varepsilon)) = 0$), then the metric entropy $h_\mu$ is upper semi-continuous with respect to $\mu$. As a result, $g$ admits a measure of maximal entropy.

\subsection{Ergodic theory for flows}
In this section we state some results on the ergodic theory for flows, which will be used later. Throughout this section, $\Lambda$ denotes a compact invariant set of the flow $\phi_t$ with singularities, and $\mu$ is a non-trivial invariant measure of $\phi_t$, i.e., $\mu(\Sing(\phi_t)) = 0$.

A dominated splitting for a flow $\phi_t$ is defined similarly to the case of diffeomorphisms. 
The set $\Lambda$ admits a {\em dominated splitting} $E \oplus F$ if this splitting is
invariant for $\Phi_t$, and there exist $C > 0$ and $\lambda < 1$ such that for every $x \in \Lambda$, and
every pair of unit vectors $u \in E_x$ and $v \in F_x$, one has
$$
\|(\Phi_t)_x(u)\| \le C\lambda^t\|(\Phi_t)_x(v)\| \mbox{ for } t > 0.
$$
We invite the readers to~\cite[Appendix B]{BDV} and~\cite{ArPa10} for more properties on the dominated splitting. The next lemma states the relation between dominated splitting for the flow and its time-one map.

\begin{lemma}
$E\oplus F$ is a dominated splitting for the flow $\phi_t|_\Lambda$ if and only if it
is a dominated splitting for the time-one map $f|_\Lambda$. Moreover, if $\phi_t|_\Lambda$ is transitive,
then we have either $X|_{\Lambda\setminus\Sing(X)} \subset E $ or $X|_{\Lambda\setminus\Sing(X)} \subset F.$
\end{lemma}
\begin{proof}
The proof of the `only if' part is trivial. Now suppose $E \oplus F$ is a
dominated splitting for $f|_\Lambda$. In order to show that it is a dominated splitting for $\phi_t$, we only need to prove that it is invariant under $\Phi_t$. 

By the commutative property between $f$ and $\phi_t$, it is easy to see that for any $t$, $\Phi_t(E)\oplus\Phi_t(F)$ is also a dominated splitting for $f$. Because the dominated splitting is unique once the dimensions of the subbundles are fixed (see~\cite[B.1.1, p.288]{BDV}), we conclude that the splitting $E \oplus F$ is invariant under $\Phi_t$. Therefore,  $E \oplus F$  is also a dominated
splitting for $\phi_t|_\Lambda$.

Now suppose $\phi_t|_\Lambda$ is transitive. Take $x \in \Lambda\setminus\Sing(X)$ such that $\Orb^+(x)$ is dense
in $\Lambda$. If $X(x) \notin  E_x \cup F_x$, then for $t$ sufficient large, $X(\phi_t(x))$ is
close to $F(\phi_t(x))$, by the domination between $E$ and $F$. We take $t_0$ large such that $\phi_{t_0}(x)$ is close to $x$, then $X(\phi_{t_0}(x))$  is close to $X(x)$ and $F_{\phi_{t_0}(x)}$ is close to $F_x$, which implies that $X(x)$
is arbitrarily close to $F_x$, a contradiction. This shows that $X(x) \in  E_x \cup F_x$. 

Because $\Orb^+(x)$ is dense, by the continuity of the flow direction and the sub-bundles $E$ and $F$, if $X(x) \in E_x$, we must have  $X|_{\Lambda\setminus\Sing(X)} \subset  E$. The same argument applies if  $X(x) \in F_x$. The proof is complete.
\end{proof}

\begin{remark}
If the dominated splitting $E\oplus F$ is sectional hyperbolic, then $\Phi_t$ on $E$ is uniformly contracting by definition. Since the flow speed $\|X(x)\|$ is bounded and thus cannot be backward exponentially expanding, we must have $X|_{\Lambda\setminus\Sing(X)} \subset F$. For more detail, see Lemma~\ref{l.domination.singularity}.
\end{remark}

\begin{definition}
{\em The topological entropy (resp. metric entropy) of a continuous flow}
is the topological entropy (resp. metric entropy) of its time-one map. A flow is $\varepsilon$-entropy expansive if its time-one map is $\varepsilon$-entropy expansive.
\end{definition}
\begin{lemma}\label{l.componententropy}
Let $\mu$ be an ergodic invariant measure of $\phi_t$, and $\tilde{\mu}$ be an ergodic
component of $\mu$ for the time-one map $f$. Then $h_\mu(\phi_t) = h_{\tilde{\mu}}(f)$.
\end{lemma}

\begin{proof}
Observe that $\tilde{\mu}_t = (\phi_t)_*\tilde{\mu}$ is also an $f$-invariant measure and
$$
\mu = \int_{[0,1]}\tilde{\mu}_t\,dt.
$$
On the other hand, $h_{\tilde{\mu}_t}(f) = h_{\tilde{\mu}}(f)$ due to the following observation: for any partition
$\cA = \{A_1, \ldots , A_k\}$, write $\cA_t = \{\phi_t(A_1), \ldots , \phi_t(A_k)\}$, then $\tilde{\mu}(A_i) =\tilde{\mu}_t(\phi_t(A_i))$.
Since the  metric entropy is an affine function with  respect to the invariant measures, we get 
$$
h_\mu(\phi_t)  = h_\mu(f) = \int_{[0,1]}h_{\tilde{\mu}_t}(f)dt = \int_{[0,1]}h_{\tilde{\mu}}(f)dt = h_{\tilde{\mu}}(f) .
$$
\end{proof}

As a corollary of the previous lemma, we state the following two results regarding entropy expansiveness in a flow version: 
\begin{lemma}\cite{B72}\label{l.variationalprinciple}
If $\phi_t$ is entropy expansive then the metric entropy function is upper semi-continuous. In particular, there exists a measure of maximal entropy.
\end{lemma}
\begin{lemma}\cite[Lemma 2.3]{LVY}\label{l.uppersemicontinuity}
Let $\cU$ be a $C^1$ open set of flows which are $\varepsilon$-entropy expansive for some $\varepsilon  > 0$. Then the topological entropy varies in an upper
semi-continuous manner for flows in $\cU$.
\end{lemma}

The {\em linear Poincar\'e flow} $\psi_t$ is defined as following: denote the normal bundle
of $\phi_t$ over $\Lambda$ by 
$$
N_\Lambda = \bigcup_{x\in\Lambda\setminus\Sing(X)}N_x,
$$
where $N_x$ is the orthogonal complement of the flow direction $X(x)$, i.e.,
$$
N_x = \{v \in T_xM: v \perp X(x)\}.
$$
Denote the orthogonal projection of $T_xM$ to $N_x$ by $\pi_x$. Given $v \in N_x$ for a regular point $x \in
M \setminus  \Sing(X)$ and recall that $\Phi_t$ is the tangent flow,  we can define $\psi_t(v)$ as the  orthogonal projection of $\Phi_t(v)$ onto $N_{\phi_t(x)}$, i.e.,
$$
\psi_t(v) = \pi_{\phi_t(x)}(\Phi_t(v)) = \Phi_t(v) -\frac{< \Phi_t(v), X(\phi_t(x)) >}{\|X(\phi_t(x))\|^2}X(\phi_t(x)),
$$
where $< .,. >$ is the inner product on $T_xM$ given by the Riemannian metric. The following is the flow version of the Oseledets theorem:
\begin{proposition}
For $\mu$ almost every $x$, there exists $k = k(x) \in \NN$ and real numbers
$$
\hat{\lambda}_1(x) > \cdots > \hat{\lambda}_k(x)
$$
and a $\psi_t$ invariant measurable splitting on the normal bundle:
$$
N_x = \hat{E}^1_x \oplus \cdots\oplus \hat{E}^k_x,
$$
such that
$$
\lim_{t\to\pm\infty}\frac1t\log \|\psi_t(v_i)\| = \hat{\lambda}λ_i(x)\mbox{ for every non-zero vector }v_i \in \hat{E}^i_x.
$$
\end{proposition}

Now we state the relation between Lyapunov exponents and the Oseledets splitting
for $\psi_t$ and for $f=\phi_1$:
\begin{theorem}
For $\mu$ almost every $x$, denote by $\lambda_1(x) > \cdots > \lambda_k(x)$ the Lyapunov
exponents and  
$$
T_xM = E^1_x \oplus\cdots\oplus E^k_x 
$$
the Oseledets splitting of $\mu$ for $f$. Then
$$
N_x = \pi_x(E^1_x)\oplus\cdots\oplus \pi_x(E^k_x)
$$ 
is the Oseledets splitting of $\mu$ for the linear Poincar\'e flow $\psi_t$. Moreover, the Lyapunov
exponents of $\mu$ (counting multiplicity) for  $\psi_t$ is the subset of the exponents
for $f$ obtained by removing one of the zero exponent which comes from the flow direction.
\end{theorem}

\begin{definition}
$\mu$ is called a {\em hyperbolic measure} for the flow $\phi_t$ if it is an ergodic
measure of $\phi_t$ and all the Lyapunov exponents for the linear Poincar\'e flow $\psi_t$ are non-vanishing. In other words, if we view $\mu$ as an invariant measure for the time-one map $f$, then $\mu$ has exactly one exponent which is zero, given by the flow direction.  We call the number of the
negative exponents of $\mu$, counting multiplicity, its \emph{index}.
\end{definition}

\section{Positive topological entropy and Entropy expansiveness}\label{s.extropyexpansive}
In this section we prove Theorem~\ref{m.h-expansive} and Theorem~\ref{m.positiveentropy}.

\subsection{Entropy expansiveness}\label{ss.expansive}
We will prove the following theorem, which is a stronger version of Theorem~\ref{m.h-expansive}.
\begin{main}\label{m.almostexpansive}
Let $\Lambda$ be a compact invariant set that is sectional hyperbolic for a $C^1$ flow $\phi_t$, with all the singularities in $\Lambda$ hyperbolic. Then there exists $\delta>0$ such that the set
$$
\Gamma_\delta: = \{x: \mbox{there exists $\tilde\delta(x)>0$ such that }B_\infty(x,\delta) \subset \phi_{(-\tilde\delta(x),\tilde\delta(x))}(x)\} 
$$
satisfies $\mu(\Gamma_\delta) = 1$ for every ergodic, invariant measure $\mu$. Furthermore, the same holds for all $\tilde\phi_t$ in a $C^1$ neighborhood of $X$ and for the maximal invariant set $\tilde{\Lambda}$ of $\tilde\phi_t$ in a small neighborhood of $\Lambda$.
 
\end{main}

To see that Theorem~\ref{m.almostexpansive} implies Theorem~\ref{m.h-expansive}, we use a  new criterion for entropy expansiveness, given in~\cite{LVY} as Proposition 2.4. For that purpose, we make the following definition:
\begin{definition}
Let $g$ be a homeomorphism on $M$. For $\varepsilon > 0$, we say that $g$ is {\em $\varepsilon$-almost entropy expansive} if for every  $g$-invariant, ergodic
measure $\mu$ and for $\mu$ almost every point $x$, we have $$h_{top}(g, B_\infty(x, \varepsilon)) = 0.$$
\end{definition}
Then~\cite[Proposition 2.4] {LVY} states that:
\begin{lemma}\label{l.almostexpansive}
g is $\varepsilon$-almost entropy expansive if and only if it is $\varepsilon$-entropy expansive.
\end{lemma}
Note that every compact flow segment has zero topological entropy since their length is bounded under the iteration of $\phi_t$. Therefore if Theorem~\ref{m.almostexpansive} holds, then for every ergodic, invariant measure $\mu$ and for $\mu$ almost every $x$, the topological entropy of $B_\infty(x,\delta)$ must be zero. Combining with Lemma~\ref{l.almostexpansive}, we see that $\phi_t$ must be robustly $h$-expansive.

The rest of this subsection is dedicated to the proof of Theorem~\ref{m.almostexpansive}. To this end, we assume that $\Lambda$ is sectional hyperbolic for the $C^1$ flow $\phi_t$ and $f=\phi_1$ is the time-one map. We also assume, as in Theorem~\ref{m.almostexpansive}, that all the singularities in $\Lambda$ are hyperbolic. We take $U$ a small neighborhood of $\Lambda$, such that the maximal invariant set $\tl$ for $\phi_t|_U$ is also sectional hyperbolic: on $T_{\tl} M$ there is a dominated splitting $E^s\oplus F^{cu}$ such that $\Phi_t$ on $E^s$ is uniformly contracting.  Moreover, there is $0<\lambda_0<1$ such that for any subspace $V_x\subset F^{cu}_x$ with dimension at least $2$, we have\footnote{Here we may take $f = \phi_{N_0}$ for some $N_0>0$ large enough if necessary. Equivalently we can replace $X$ by $cX$ for some $c>0$ large enough. Either way, it will affect the definition of $B_\infty(x,\delta)$ since it is defined using the time-one map. However, by continuity we have $B_\infty(x, \delta', f)\subset B_\infty(x,\delta, f^{N_0})$  for $\delta' = \delta'(\delta, N_0)$ small enough. Therefore if Theorem~\ref{m.almostexpansive} is proven for $f = \phi_{N_0}$ then it also holds for $\phi_1$ by taking a smaller $\delta$.}
$$
\det (Df|_{V_x}) > \frac{1}{\lambda_0}.
$$ 
Enlarging $\lambda_0$ if necessary, we can assume that the above inequality holds for any two-dimensional subspace $V_x$ in the cone $C_a(F^{cu})$, for $a$ small enough. Since all the singularities in $\Lambda$ are hyperbolic and thus isolated, we can take $U$ small enough so that 
\begin{equation}\label{e.singularities}
\Sing(\phi_t|_U)= \Sing(\phi_t|_{\Lambda}).
\end{equation}

Note that  for a $C^1$ flow $\tilde\phi_t$ close to $\phi_t$, the maximal invariant set of $\tilde\phi_t$ in $U$ is still sectional hyperbolic, with all singularities hyperbolic. Below we will only show the $\delta$-entropy expansiveness for the flow $\phi_t$. For the robustness, one can easily check that the choice of $\delta$ depends only on the fake foliation which is continuous with respect to the system (see \cite{LVY}), the flow speed, the hyperbolicity of the singularities and the volume expanding rate $\lambda_0$ (see Propositions~\ref{p.measure.away} and~\ref{p.measure.near} below), thus can be made uniform for nearby flows.

\subsubsection{Structure of the proof}\label{sss1}
Before getting into details, we  briefly explain the structure of our proof of Theorem~\ref{m.almostexpansive}. First we  introduce the fake foliations for maps with dominated splitting. As we will see later, these fake foliations are $f$-invariant but generally not $\phi_t$ invariant for non-integers $t$. In particular, the $cu$ fake leaves are not saturated by the flow orbits. However, there is a weak form of saturation for points in the infinite Bowen ball, as observed in Corollary~\ref{c.saturation}.

To prove Theorem~\ref{m.almostexpansive}, the key observation is that the infinite Bowen ball $B_\infty(x,\delta)$ must be contained in the fake $cu$-leaf of $x$ (Lemma~\ref{l.bowenball}). Moreover, for $\mu$ almost every $x$, the distance between $x$ and $y$ are exponentially expanding under the flow. Thus the $(\infty,\delta)$-Bowen ball of $x$ must be contained in the orbit of $x$.

The key result here  is Lemma~\ref{l.expansion.near}, which builds up the expanding property on the normal direction, for orbits that pass through the $\varepsilon$ neighborhood of a singularity (this lemma also gives the choice of $\delta$). Then we consider the following two families of measures:
\begin{itemize}
\item  measures whose supports are $\varepsilon$ away from any singularity. For typical points of such measures, the orbits stay away from singularities, thus the flow speed is bounded from below. In this case, the sectional hyperbolicity guarantees that there is enough expansion along fake $\cF^{cu}$ foliation (Lemma~\ref{l.expansiong.away}).
\item measures whose supports intersect the $\varepsilon$ neighborhood of some singularity. The orbit of typical points will pass through the $\varepsilon$ neighborhood of singularities infinitely many times. We use Lemma~\ref{l.expansion.near} to establish the expansion behavior for each time a point gets close to a singularity. This method allows us to bypass linearization near singularities.  
\end{itemize}

From now on, to simplify notation, we will write $x_t = \phi_t(x)$. In particular, 
$$
x_n = f^n(x) = \phi_n(x).
$$


\subsubsection{Fake foliations and infinite Bowen balls}

The following lemma is borrowed from \cite[Lemma 3.3]{LVY} (see also \cite[Proposition 3.1]{BW}), which shows that one can always construct local fake foliations. Moreover, these fake foliations have local product structure, which is preserved by the dynamics.

\begin{lemma}\label{l.fakefoliation}
Let $K$ be a compact invariant set of $f$. Suppose $K$ admits a dominated splitting $T_K M = E^1 \oplus E^2 \oplus E^3$. Then for every $a>0$ there are $\rho > r_0 > 0$, such that the neighborhood $B_\rho(x)$ of every $x \in K $ admits foliations $\cF^1_x, \cF^2_x, \cF^3_x, \cF^{12}_x$ and $\cF^{23}_x$, such that for every $y \in B_{r_0}(x)$ and $* \in \{1, 2, 3, 12, 23\}:$
\begin{enumerate}[label=(\roman*)]
\item $\cF^*_x(y)$ is $C^1$ and tangent to the respective cone $C_a(E_x^*)$ (we write $E^{12} = E^1\oplus E^2$, similarly for $E^{23}$).
\item Forward and backward invariance: $f(\cF^*_x(y, r_0)) \subset \cF^*_{f(x)}(f(y))$, and\\ $f^{-1}(\cF^*_x(y, r_0)) \subset \cF^*_{f^{-1}(x)}(f^{-1}(y))$.
\item $\cF^1_x$ and $\cF^2_x$ sub-foliate $\cF^{12}_x$; $\cF^2_x$ and $\cF^3_x$ sub-foliate $\cF^{23}_x$.
\end{enumerate}
\end{lemma}
Now we take $K=\tl$, and consider the fake foliations $\cF^s$ and $\cF^{cu}$ given by the dominated splitting $E^s\oplus F^{cu}$. Note that the forward and backward invariance above may not hold for the flow, i.e., the fake foliation
may not be preserved by $\phi_t$ when $t \notin \ZZ$. Moreover, the flow orbits may not even locally
saturate $\cF^{cu}$ leaves. This is because the fake foliations depend on the extension of the
dynamics in the tangent bundle (see the proof of \cite[Proposition 3.1]{BW}), which is in general not
preserved by the flow.

For the convenience of our readers, we provide a list of parameters that will be used in this section:

\begin{enumerate}
\item $\varepsilon_0$: the scale within which the exponential map is well-defined.
\item $\lambda_0<1$: $\frac{1}{\lambda_0}$ is the volume expanding rate on $F^{cu}$.
\item $0<r_0<\rho$: the fake foliations are defined and invariant within $r_0$ neighborhood of every $x\in\tl$, which foliate $\rho$ neighborhoods. Also we have $B_\infty(x,r_0) \subset \cF^{cu}_x(x)$ by Corollary~\ref{c.saturation} below. 
\item Let $D_0$ be the maximal flow speed and $r_1=\frac{r_0}{2D_0}$. 
\item Fixed some $\delta_1>0$ small, we get $L>1$ whose precise definition is in the proof of Proposition~\ref{p.measure.near}, then $\varepsilon$, $\delta>0$ is given by Lemma~\ref{l.expansion.near}. Note that both $\varepsilon$ and $\delta$ are chosen to be much smaller than $r_0$, the size of the fake foliations.
\end{enumerate}

The next lemma is taken from the proof of~\cite[Theorem 3.1]{LVY} and gives an important observation on infinite Bowen balls. The proof easily follows from the local product structure of the fake foliations, and the uniform contraction on $\cF^s$. Unlike in~\cite{LVY}, here we do not need the Pliss Lemma. 

\begin{lemma}\label{l.bowenball} Let $r_0>0$ be given as in Lemma~\ref{l.fakefoliation}.
Then for every $x \in \tl$, $B_\infty(x,r_0) \subset \cF^{cu}_x(x)$.
\end{lemma}
\begin{proof}
Let $y\in B_{\infty}(x,r_0)$ and write $y'$ the unique point of intersection of $\cF^s_x(x)$ with $\cF^{cu}_x(y)$. We claim that $x=y'$, which implies that $y\in \cF^{cu}_x(x)$. 

Since $f^{-n}(y)\in B_{r_0}(f^{-n}(x))$ for all $n\ge 0$, by the invariance of the fake foliation, we have $f^{-n}(y')\in \cF^s_{f^{-n}(x)}(f^{-n}(x), r_0)$ for all $n\ge 0$, which means that
$$
y' \in \cap_{n\ge 0} f^n\left(\cF^s_{f^{-n}(x)}(f^{-n}(x), r_0)\right)
$$

On the other hand, since $E^s$ is uniformly contracted by $Df$, we can take $a$ small enough such that vectors in $C_a(E^s)$ are contracted by some $\lambda'<1$ under the iteration of $Df$.
This shows that 
$$
y'\in \cF^{s}_{x}(x, r_0(\lambda')^n )
$$
for all $n\ge 0$. As a result  we must have $y' = x$, as claimed.
\end{proof}
As an immediately corollary of Lemma~\ref{l.bowenball}, we obtain a type of weak saturated
property for $cu$ fake leaf $\cF^{cu}_x(x)$:
\begin{corollary}\label{c.saturation}
For $r_1 = r_0/(2D_0)$ and for every $x\in M$, $y \in B_\infty(x, r_0/2)$, we have $y_t \in \cF^{cu}_x(x)$
for $|t| \le r_1$.
\end{corollary}
\begin{proof}
Denote $D_0= \max_{x\in M} \|X(x)\|$ and $r_1 =\frac{r_0}{2D_0}$. Then for any $|t| \le r_1$ and every $n\in\ZZ$, the segment of flow orbit between $f^n(y_t)$ and $f^n(y)$ has length bounded by $r_1D_0 \le r_0/2$.
Hence, for every $n\in\ZZ$,
$$
d(f^n(x), f^n(y_t)) \le d(f^n(x), f^n(y)) + d(f^n(y), f^n(y_t)) \le r_0,
$$
which shows that $y_t \in B_\infty(x, r_0) \subset \cF_x^{cu}(x)$ for $|t| \le r_1$.

\end{proof}

\subsubsection{Expanding property on $cu$-leaves}
In this section we will present two lemmas (\ref{l.expansiong.away} and~\ref{l.expansion.near}) which establish the expanding property between different flows lines for points in $B_\infty(x,\delta)$, for some $\delta \ll r_0$. The proof of these lemmas will be postponed to the end of this section.
First we introduce the following definition:
\begin{definition}
For any $x\in U\setminus\Sing(X)$ and $0<\delta<\varepsilon_0$, write $B^\perp_{\delta}(x) = \exp_x(N_x\cap B_\delta(0))$ for the image of local $N_x$ under the exponential map. Denote by $d_x^*$ the distance between $x$ and $y$ in the submanifold $B^\perp_{\delta}(x)\cap\cF^{cu}_x(x)$. Also write $P_x$ for the projection along the flow:
$$
P_x:B_\delta(x) \to B^\perp_{\delta}(x),
$$
which is well-defined in a small neighborhood of $x$ (though the size of the neighborhood depends on $x$), and sends every point $y$ along the flow direction to the normal plane $B^\perp_\delta(x)$. For points $y$ in  this neighborhood, write $t_x(y)$ the time for which 
$$
y_{t_x(y)} = P_x(y).
$$
\end{definition}
Note that $t_x$ becomes unbounded as $x$ gets closer to some singularity. On the other hand, for every fixed $\varepsilon>0$, we can take $\delta$ small enough (depending on $\varepsilon$), such that for every $x\notin B_\varepsilon(\Sing(X))$, $t_x$ is uniformly small inside $B_\delta(x)$. 
In view of Corollary~\ref{c.saturation}, we take $\delta$ small such that $|t_x(y)|<r_1$ for $y\in B_\delta(x)$ and $x\notin B_\varepsilon(\Sing(X))$. As a result, for $y\in B_\infty(x,\delta)$, $P_x(y)$ is still contained in $\cF_x^{cu}(x).$
The proof of the next lemma is straightforward and thus omitted.
\begin{lemma}\label{l.comparedistance}
For any $\varepsilon>0$ and $L_0>1$, there is $\delta>0$ such that for every $x\in\tl\setminus B_\varepsilon(\Sing(X))$ and $y\in B^\perp_{\delta}(x)\cap\cF^{cu}_x(x)$, we have
$$
d(x,y)\le d^*_x(x,y) \le L_0d(x,y).
$$
\end{lemma}

The following lemma considers points {\em{whose orbit stay $\varepsilon$-away from all singularities}}:
\begin{lemma}\label{l.expansiong.away}
For $\varepsilon > 0$ and $1 < b_0 < \frac{1}{\lambda_0}$, there is $\delta' > 0$ such that for any $x$
satisfying $(x_t)_{t\in[0,1]} \subset \tl \setminus B_\varepsilon(\Sing(X))$, and for any $y \in B_\infty(x, \delta')$:
$$
d^*_{x_1}(x_1, P_{x_1}(y_1)) > b_0\frac{\|X(x)\|}{\|X(x_1)\|}d^*_x(x, P_x(y)).
$$
\end{lemma}
In other words, Lemma \ref{l.expansiong.away} states that for points $y\in B_\infty(x, \delta')$ , one sees an expansion by a factor of $b_0\frac{\|X(x)\|}{\|X(x_1)\|}$ along the normal direction under the iteration of $f$, as long as the orbit $\{\phi_t(x)\}_{t\in[0,1]}$ stays $\varepsilon$ away from all singularities.

To estimate the expanding property for points travelling near a singularity, we take $\delta_0>0$ small enough (the choice of $\delta_0$ will be made clear in Remark~\ref{r.fakefoliation}). 
For each $\sigma\in\Sing(\phi_t|_\Lambda)$ and $\delta \ll \delta_0$, we consider the set:
$$
J_\delta(\sigma) = \{y: y\in B_{\delta_0}(\sigma)\mbox{ and } \|X(y)\| \le \delta.\}
$$
Next lemma establishes the expansion within $B_\infty(x,\delta'')$ for points travelling through the $\varepsilon$ neighborhood of a singularity. 

\begin{lemma}\label{l.expansion.near}
For every $L>1$ and $\delta_1>0$ small enough, there are constants $\varepsilon$, $\delta''>0$ with $\varepsilon \ll \delta_1$, such that for each singularity $\sigma$, $x\in B_{\delta_0}(\sigma)\cap\tl$ and $T>0$ that satisfies
\begin{itemize}
\item $\{x_t\}_{t\in[0,T]} \subset B_{\delta_0}(\sigma)$ and $\{x_t\}_{t\in[0,T]} \cap B_\varepsilon(\sigma)\ne\emptyset$; 
\item $x,x_T \notin J_{\delta_1}(\sigma), \{x_t\}_{t\in[0,1]}\cap J_{\delta_1}(\sigma) \ne \emptyset$ and $\{x_t\}_{t\in[T-1,T]}\cap J_{\delta_1}(\sigma) \ne \emptyset$,
\end{itemize}
then for every $y\in B_\infty(x,\delta'')$, we have
$$
d^*_{x_T}(x_T,P_{x_T}(y_T))>Ld^*_x(x,P_x(y)).
$$
\end{lemma}
See Figure~\ref{f.2} on Page~\pageref{f.2}. Roughly speaking, this lemma states the following: given $\delta_1>0$, one can always take $\varepsilon\ll\delta_1$ small enough, such that if the orbit of $x$ starts with $\|X(x)\|\approx \delta_1$, gets $\varepsilon$-close to a singularity, then leaves the $\varepsilon$-neighborhood of the singularity to a position $x_T$ where $\|X(x_T)\|\approx\delta_1$ (note that $\|X(x)\|$ and $\|X(x_T)\|$ are ``comparable'', and the flow speed in $B_\varepsilon(\sigma)$ is much smaller than $\delta_1$), then on the orbit segment from $x$ to $x_T$, one picks up an expanding factor of $L>1$ on the normal direction inside the $\cF^{cu}$ leaf.

\subsubsection{Measures $\varepsilon-$away from singularities}
For now we will assume that Lemma~\ref{l.expansiong.away} and~\ref{l.expansion.near} hold. Fix $L>1$, $\delta_1>0$ and take $\varepsilon$ according to Lemma~\ref{l.expansion.near}, and
consider ergodic measures whose supports are $\varepsilon$ away from any singularity. We will prove that for $\delta'>0$ given by Lemma~\ref{l.expansiong.away}, and for {\em every} points of such measures, the $(\infty,\delta')-$Bowen ball is degenerate. Recall that all the singularities of $\phi_t$ in $U$ are exactly those contained in $\Lambda$, all of which are hyperbolic, and thus isolated. Write $\Sing(\phi_t|_{\Lambda}) = \{\sigma_i\}_{i=1}^k$.
\begin{proposition}\label{p.measure.away}
There is $\delta',$ $K_1>0$ and $\tilde{\delta}>0$, such that if $\mu$ is an invariant, ergodic measure which satisfies
$$
\mu(B_\varepsilon(\sigma_i)) = 0 \mbox{ for all } i=1,2,\ldots,k,
$$
then for every $x\in\supp\mu$, $B_\infty(x,\delta') \subset \phi_{(-\tilde{\delta}, \tilde{\delta})}(x)$ with length bounded by $K_1$.
\end{proposition}

\begin{proof} 
Since $\supp\mu$ is $\varepsilon$ away from all singularities, so are all points $x\in\supp\mu$. For every $x\in\supp\mu$, apply Lemma~\ref{l.expansiong.away} on $\{x_i:i\in\NN\}$ yields
$$
d^*_{x_n}(x_n,P_{x_n}(y_n))>b_0^n\frac{\|X(x)\|}{\|X(x_n)\|}d^*_x(x,P_x(y)),
$$
for every $y\in B_\infty(x,\delta')$ and every $n>0$.

The  term $\frac{\|X(x)\|}{\|X(x_n)\|}$ is bounded from above and below since $\{x_t\}$ stays away from singularities. $d^*_{x_n}(x_n,P_{x_n}(y_n))$ is also bounded since $y\in B_\infty(x,\delta')$. Therefore $d^*_x(x,P_x(y))$ must be $0$, so is $d(x,P_x(y))$, which shows that $y$ is indeed contained in the local orbit of $x$. One can take  $\tilde{\delta}>0$ and $K_1>0$ uniform in $x\in\supp\mu$, such that every connected component of $\Orb(x)\cap B_{\delta'}(x)$ has length bounded by $K_1$ and flow time bounded by $\tilde{\delta}$. This concludes the proof of Proposition~\ref{p.measure.away}
\end{proof}


\subsubsection{Measures near singularities and proof of Theorem~\ref{m.almostexpansive}}\label{sss.singularity}

We have shown that if $\mu$ is a measure whose support is $\varepsilon$ away from singularities, then every point $x$ in the support of $\mu$ has degenerate  infinite Bowen ball,  i.e., the infinite Bowen ball is reduced to an orbit segment. It remains to consider measures whose support intersects the $\varepsilon$ neighborhood of some singularity.

Recall that $\delta'$ is given by Lemma~\ref{l.expansiong.away} and $\delta''$ by Lemma~\ref{l.expansion.near}. The main proposition in this subsection is the following:

\begin{proposition}\label{p.measure.near}
Let $\delta=\min\{\delta',\delta''\}$. For every invariant, ergodic measure $\mu$ on $\tl$ with $B_\varepsilon(\Sing(X))\cap \supp\mu\ne\emptyset$ and for $\mu$ almost every point $x$, the infinite Bowen ball $B_\infty(x, \delta)$ is a segment of the orbit of $x$.
\end{proposition}
This will conclude the proof of Theorem~\ref{m.almostexpansive}, and Theorem~\ref{m.h-expansive} will follow from Lemma~\ref{l.almostexpansive} with $\delta = \min\{\delta',\delta''\}$. Note that $\delta'$ and $\delta''$ only depend on the hyperbolicity of the singularities and the sectional volume expanding rate, and thus can be made continuous with respect to the flow $X$. This  gives the robust $\delta$-entropy expansiveness.

\begin{proof}[Proof of Proposition~\ref{p.measure.near}]

Recall that $1<b_0<\frac1{\lambda_0}$. We can shrink $\delta_1$ and $\varepsilon$ if necessary, such that for every $1\le i, j \le k$ ($i$ and $j$ could be equal) and for each $x\in B_{\varepsilon}(\sigma_i)$, if we write $T = \inf \{t>0: \phi_T(x)\notin J_{\delta_1}(\sigma_i)\}$ and $T' = \inf \{t>0: \phi_t(x)\in J_{\delta_1}(\sigma_j)\}$ then we have $T' - T > 3$.

Let 
\begin{align*}\numberthis\label{n1}
D' =\inf_{x,y}\Bigg\{ &\frac{\|X(x)\|}{\|X(y)\|}: \mbox{ there are singularities } \sigma, \sigma' \mbox{ such that } \\
&x \in \{\phi_t(\partial(J_{\delta_1}(\sigma)))\}_{t\in[0,1]} \mbox{ and }y \in \{\phi_t(\partial(J_{\delta_1}(\sigma')))\}_{t\in[-1,0]}\Bigg\}.
\end{align*}
Note that $D'\le 1$. Let $L = \frac{b_0}{D'}>1$ and   $\varepsilon,\delta'$ be given by Lemma~\ref{l.expansion.near}. Let $\delta''$ be given by Lemma~\ref{l.expansiong.away} using the $\varepsilon$ above, and take $\delta = \min\{\delta',\delta''\}$. We verify that $\delta$ satisfies Proposition~\ref{p.measure.near}.

Let $\mu$ be a non-trivial ergodic measure on $\tl$ such that $B_\varepsilon(\Sing(X))\cap \supp\mu\ne\emptyset$. Let $x$ be a typical point of $\mu$. Then the orbit of $x$ must visit the $\varepsilon$ neighborhood  of singularities infinitely many times. The idea of the proof is very simple: we use Lemma~\ref{l.expansion.near} to get expansion for each time the orbit travels through the $\varepsilon$-neighborhood of $\Sing(X)$, and use Lemma~\ref{l.expansiong.away} to control the expansion in-between.

To this end we define a sequence $0\le T_1\le T_1'<T_2<T_2'<\ldots$, such that for each $n$, we have 
\begin{enumerate}
\item The time interval $[T_n,T_n']$ contains the orbit segment travelling through the $\varepsilon$-neighborhood:
$\{x_t\}_{[T_n,T_n']}\subset B_{\delta_0}(\sigma)$ and $\{x_t\}_{[T_n,T_n']}\cap B_\varepsilon(\sigma)\ne\emptyset$, for some singularity $\sigma$.
\item The start and end point of $\{x_t\}_{[T_n,T_n']}$ have ``comparable" flow speed:\\ $x_{T_n},x_{T_n'} \notin J_{\delta_1}(\sigma), \{x_t\}_{t\in[T_n,T_n+1]}\cap J_{\delta_1}(\sigma) \ne \emptyset$ and $\{x_t\}_{t\in[Tn'-1,Tn']}\cap J_{\delta_1}(\sigma) \ne \emptyset$.
\item The orbit segment corresponding to the time interval $[T_n',T_{n+1}]$ are  outside $\varepsilon$-neighborhood of singularities: $\{x_t\}_{t\in[T'_n,T_{n+1}]} \cap B_\varepsilon(\Sing(X)) = \emptyset$.
\end{enumerate}
By Lemma~\ref{l.expansiong.away} we obtain for  $y\in B_\infty(x,\delta)$,
\begin{align*}
d^*_{x_{T_{n+1}}}(x_{T_{n+1}}, P_{x_{T_{n+1}}}(y_{T_{n+1}}))&>b_0^{T_{n+1}-T_n'}\frac{\|X(x_{T_n'})\|}{\|X(x_{T_{n+1}})\|}d^*_{x_{T'_{n}}}(x_{T'_{n}}, P_{x_{T'_{n}}}(y_{T'_{n}}))\\
&>b_0^{T_{n+1} - T_n'} D' d^*_{x_{T'_{n}}}(x_{T'_{n}}, P_{x_{T'_{n}}}(y_{T'_{n}})).
\end{align*}
Applying Lemma~\ref{p.measure.near} on the time interval $[T_n,T_n']$ yields
$$
 d^*_{x_{T'_{n}}}(x_{T'_{n}}, P_{x_{T'_{n}}}(y_{T'_{n}})) > \frac{b_0}{D'} d^*_{x_{T_{n}}}(x_{T_{n}}, P_{x_{T_{n}}}(y_{T_{n}})).
$$
Inductively we get:
$$
 d^*_{x_{T'_{n}}}(x_{T'_{n}}, P_{x_{T'_{n}}}(y_{T'_{n}}))> b_0^n d^*_{x_{T_{1}}}(x_{T_{1}}, P_{x_{T_{1}}}(y_{T_{1}})),
$$
but the left hand side is bounded since $y\in B_\infty(x,\delta)$, thus $d^*_{x_{T_{1}}}(x_{T_{1}}, P_{x_{T_{1}}}(y_{T_{1}}))=0$, that is, $y$ belongs to the local orbit of $x$.

This proves Proposition~\ref{p.measure.near}, and Theorem~\ref{m.almostexpansive} now follows from Lemma~\ref{l.almostexpansive}, Proposition~\ref{p.measure.away} and~\ref{p.measure.near}.
\end{proof}


\subsubsection{Proof of Lemma~\ref{l.expansiong.away} and~\ref{l.expansion.near}}
\begin{proof}[Proof of Lemma~\ref{l.expansiong.away}]

The idea of the proof is very simple: we consider the `parallelogram' generated by $X(x)$  and the vector joining $x$ and $P_x(y)$, whose area is approximately  $\|X(x)\|\cdot d^*_x(x, P_x(y))$. Then we compare it with the `parallelogram' generated by $X(x_1)$ and the vector joining $x_1$ and $P_{x_1}(y_1)$, which has area approximately $\|X(x_1)\|\cdot d^*_{x_1}(x_1, P_{x_1}(y_1)) $. The expanding factor $b_0$ is then given by the sectional hyperbolicity on $F^{cu}$.

To this end write $C = (\frac{1}{\lambda_0b_0})^{1/6}$ and take $0 < t_0 < \min\{r_0/2,r_1/2\}$ small enough, such that for any $2$-dimensional subspace $\Sigma$ in the tangent space, we have $1/C < \Jac(\Phi_t|_\Sigma) < C$ for any $|t| < t_0$.

For two vectors $u, v \in T_xM$, denote by $\cP[u, v]$ the parallelogram defined by these two vectors and $A(u, v)$ its area. For $\delta > 0$ sufficiently small, for $y \in B_\infty(x, \delta)$, we have  $|t_x(y)|, |t_{x_1}(y_1)| < t_0< r_1$. Then by Corollary~\ref{c.saturation}, $P_x(y) \in \cF^{cu}_x(x) \cap B^\perp_{r_0}(x)$ and similar relation holds for $P_{x_1}(y_1)$.

There is $0 < \varepsilon_1 \ll r_0/2$ depending on $\varepsilon$ and $b_0$, such that for any $z \in B^\perp_{\varepsilon_1}(x_t)$, the following conditions are satisfied:
\begin{enumerate}
\item  $\frac1C <\frac{\|X(z)\|}{\|X(x_t)\|} < C$.
\item For $v \in T_zB^\perp_{\varepsilon_1}(x_t)$, we have 
$\frac1C\|v\|\|X(z)\| \le A(v, X(z)) \le C\|v\|\|X(z)\|$.
\end{enumerate}
We may further suppose that $\delta$ is sufficiently small, such that 
$d^*_x(x, P_x(y)) < \varepsilon_1,d^*_{x_1}(x_1, P_{x_1}(y_1)) < \varepsilon_1 $, and by Lemma~\ref{l.comparedistance}:
$$
d(x, P_x(y)) < d^*_x(x, P_x(y)) < Cd(x, P_x(y)),
$$
and the same holds for $x_1$ and $y_1$. See Figure \ref{f.1}.
\begin{figure}[h]
	\centering
	\includegraphics[scale=1]{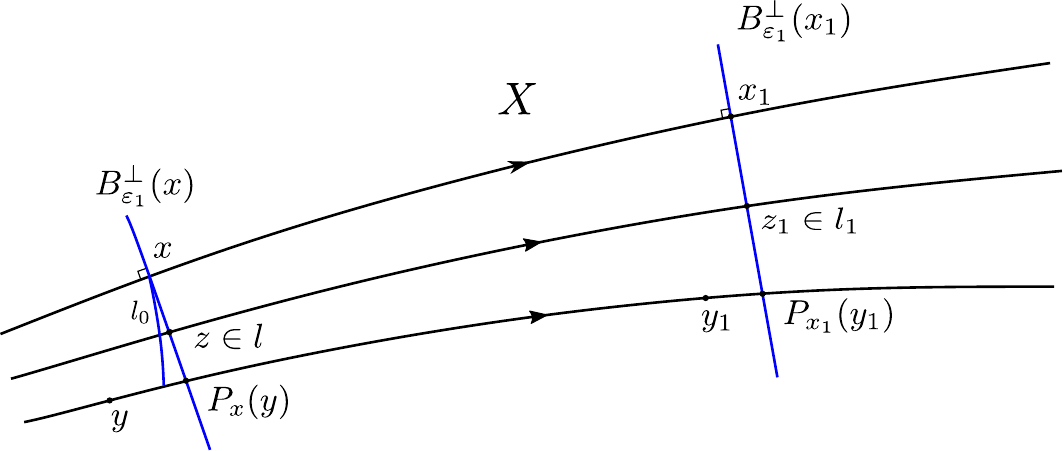}
	\caption{}
	\label{f.1}
\end{figure}

Take a curve $l_1 \subset  B^\perp_{\varepsilon_1}(x_1) \cap \cF^{cu}_{x_1}(x_1)$ which links $x_1$ and $P_{x_1}(y_1)$ with $\length(l_1) = d^*_{x_1}(x_1, P_{x_1}(y_1)))$, and write $l_0 = f^{-1}(l_1) $ and $l = P_x(l_0)$, we may suppose
$|t_x|_{l_0}| \le t_0$ and $l \subset B^\perp_{\varepsilon_1}(x)$. Then $l$ is a curve contained in $B^\perp_{\varepsilon_1}(x)$ which connects $x$ and $P_x(y)$. We note that although $l_0\subset \cF^{cu}_x(x)$ by the local invariance of fake leaves, $l$ is not necessarily contained in $\cF^{cu}_x(x)$ (recall that the saturation property only applies to points in the Bowen ball). Denote $H$ the smooth holonomy map between $l_1$ and $l$ which is induced by flow, then $H(P_{x_1}(y_1)) = P_x(y)$. We claim that
$$
\|DH\| \le \lambda_0C^5\frac{\|X(x_1)\|}{\|X(x)\|},
$$
which implies that
$$
d(x, P_x(y)) \le \length(l) \le \lambda_0C^5\frac{\|X(x_1)\|}{\|X(x)\|}d^*_{x_1}(x_1, P_{x_1}(y_1)).
$$
Then the lemma follows by changing  $d(x, P_x(y))$ to $d^*_{x}(x, P_{x}(y))$, resulting in an extra power of $C$, and our choice of $C = (\frac{1}{\lambda_0b_0})^{1/6}$.

It remains to prove this claim. 

For any $z_1 \in l_1 $, denote by $v_1$ a tangent vector of $l_1$, write $z_0=f^{-1}(z_1)$ and $z  = \phi_{t_x(z_0)}(z_0) \in l$. Also write $v_0 = DH^{-1}(v_1)$.

Then by sectional-hyperbolic,
$$
A(\Phi_{-1}(v_1), X(z_0)) \le \lambda_0 A(v_1, X(z_1)).
$$

Because $|t_x|_{l_0}| < t_0$, by the assumption on $t_0$,

\begin{equation}\label{e.area1}
A(\Phi_{t_x(z_0)}\left(\cP[\Phi_{-1}(v_1), X(z_0)]\right))\le  C\lambda_0A(v_1, X(z_1)).
\end{equation}
Since $\Phi_{t_x(z_0)}\Phi_{-1}(v_1)= \Phi_{-1+t_x(z_0)}(v_1)$ and $\Phi_{t_x(z_0)}(X(z_0)) = X(z)$, we have\begin{equation}\label{e.area2}
A(\Phi_{t_x(z_0)}\left(\cP[\Phi_{-1}(v_1), X(z_0)]\right))= A(\Phi_{-1+t_x(z_0)}(v_1),X(z)).
\end{equation}
Note that $v_0=DH^{-1}(v_1)$ is the projection of $\Phi_{-1+t_x(z_0)}(v_1)$ along $X(z)$ on $T_zB^\perp(x)$,
combine equations (\ref{e.area1}), (\ref{e.area2}), we obtain
\begin{equation}
A(v_0, X(z)) = A(\Phi_{-1+t_x(z_0)}(v_1),X(z))\le C\lambda_0A(v_1, X(z_1)).
\end{equation}
By the assumptions (1) and (2) above on $\varepsilon_1$, we get
$$
C^{-2}\|v_0\|\|X(x)\| \le C^3\lambda_0\|v_1\|\|X(x_1)\|,
$$
which implies the  desired claim:
$$
\|DH\| \le \frac{v_0}{v_1} \le C^5 \lambda_0 \frac{\|X(x_1)\|}{\|X(x)\|}.
$$
\end{proof}

To prove Lemma~\ref{l.expansion.near}, we start by showing that there is a finer dominated splitting on $\Sing(\phi_t|_U)$. Recall that we take $U$ small enough, such that $\Sing(\phi_t|_U)= \Sing(\phi_t|_{\Lambda}) = \{\sigma_1,\ldots,\sigma_k\}$. 

\begin{lemma}\label{l.domination.singularity}
For every $j=1,\ldots,k$, $Df|_{F^{cu}_{\sigma_j}}$ has exactly one eigenvalue with norm less than one. As a result, there is a hyperbolic splitting $E^s\oplus E^c\oplus E^u$ on $\Sing(\phi_t|_U)$, with $E^c \oplus E^u = F^{cu}$.
\end{lemma}
\begin{proof}
The proof is quite standard. Suppose that for some $\sigma \in \Sing(\phi_t|_U)$, all the eigenvalues of $Df|_{F^{cu}_{\sigma}}$ are positive (recall that we assume all singularities to be hyperbolic), we claim that $\Lambda\cap W^s(\sigma)\setminus\{\sigma\} \ne\emptyset$.

To prove this claim, we take $x_n \in \tl$ and $t_n\to\infty$, such that $\phi_{t_n}(x_n)\to\sigma$ as $n\to\infty.$ Fix some $\varepsilon>0$ small enough so that there is no singularity in $B_\varepsilon(\sigma)$ other than $\sigma$, and let $y^n = \phi_{t_n'}(x_n)$ where $t_n'<t_n$  is the last time the orbit of $x_n$ enters $B_\varepsilon(\sigma)$. Taking subsequence if necessary, we may assume that $y^n \to y^0\in\tl \setminus \Sing(X)$. Since $\phi_t(y^0)\in B_\varepsilon(\sigma)$ for all $t>0$, this shows that $y^0\in W^s(\sigma)\setminus \{\sigma\}$.

By the invariance of $W^s(\sigma)$, $\phi_t(y^0) \in W^s(\sigma)$ for all $t\in \RR$; in particular, $X(y^0)$ is contained in the $E^s$ cone. Fix some $\delta>0$ small enough and consider the orbit segment $l=\{\phi_t(y^0)\}_{t\in[-\delta,\delta]}$, it follows that $l$ is tangent to the $E^s$ cone. Since $f^{-1}$ expands vectors in $E^s$ cone, we have $\length(f^{-n}(l))/\length(l) \to \infty$, which contradicts the a priori estimate:
$$
\length(f^{-n}(l)) \le \frac{\max\{\|X(x)\|:x\in\Lambda\}}{\min\{\|X(y)\|:y\in l\}}\length(l),
$$
which is bounded.
\end{proof}

\begin{remark}\label{r.fakefoliation}
In the previous section we defined the fake foliations $\cF^s, \cF^{cu}$ using the dominated splitting $E^s\oplus F^{cu}$ on $\tl$. These two foliations are defined around the $r_0$ ball at every $x\in \tl$, in particular, around singularities. On the other hand, the hyperbolic splitting $E^s\oplus E^c\oplus E^u$ (extended to a small neighborhood of singularities) gives fake foliations $\cF_S^i$, $i=s,c,u,cs,cu$, which only exists near a neighborhood of singularities. We take $\delta_0>0$ small enough such that for every $\sigma\in\Sing(\phi_t|_\Lambda)$,
\begin{itemize}
\item $\sigma$ is the only singularity within $B_{\delta_0}(\sigma)$.
\item $\delta_0 \ll r_0/2$, such that both the fake foliations $\cF^s,\cF^{cu}$ and the (also fake) foliations $\cF_S^i$, $i=s,c,u,cs,cu$ are well-defined within $B_{\delta_0}(\sigma)$.
\end{itemize}
The reason that we still need fake foliations $\cF^{cu}$ for points in $B_{\delta_0}(\sigma)$ is due to Lemma~\ref{l.bowenball}.
\end{remark}

\begin{proof}[Proof of Lemma~\ref{l.expansion.near}]
First let us quickly explain the structure of the proof. The fake foliations $\cF^*$ and $\cF_S^*$ give a local product structure within  $B_{\delta_0}(\sigma)$, which allows us to consider the $u$ and $c$ distance between two points. If we take such points close to $W^u(\sigma)$ (that is, their orbits are about to leave $J_{\delta_1}(\sigma)$), then the $u$-distance will be contracted by $\phi_{-t}$ exponentially fast, as the $E^u(\sigma)$ bundle is uniformly expanding. On the other hand, the $c$-distance seems to be expanding under backward iteration, due to the $E^c(\sigma)$ bundle having negative exponent. However, when projected to the normal plane $N_x$, the $c$-segment will be contracting under $\phi_t$ because of the sectional hyperbolicity.  This shows that the distance on the normal plane is contracted under backward iteration. See Figure \ref{f.2}, and note that the entire figure is contained in $B_{\delta_0}(\sigma)$ which is a much bigger ball.

\begin{figure}[h]
	\centering
    \includegraphics[scale=0.4]{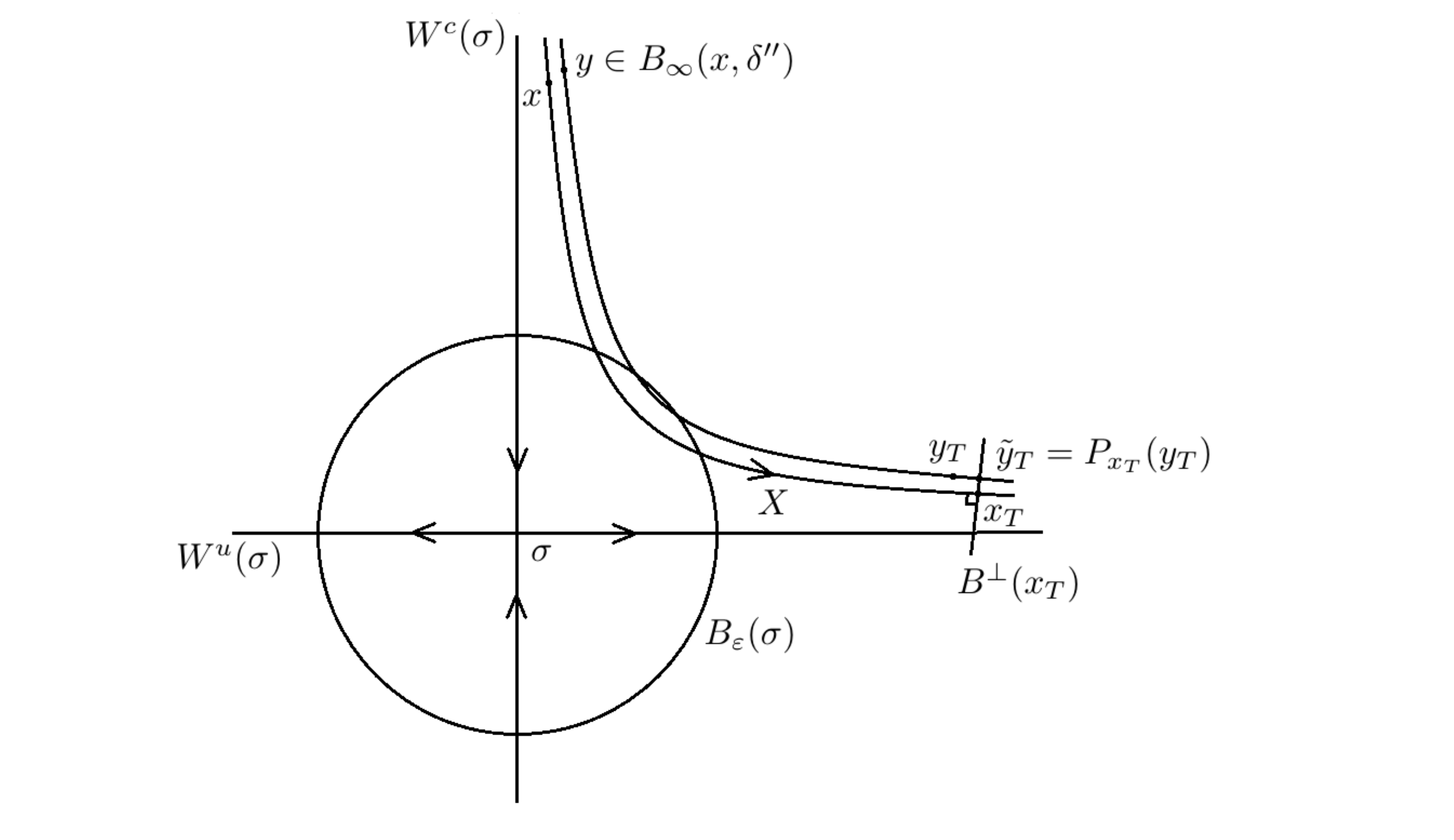}
    \caption{The orbit segment $\{x_t\}_{t\in [0,T]}$.}
    \label{f.2}
\end{figure}
From now on, without loss of generality we will assume that $T$ is an integer. 

To prove expansion for $y\in B_\infty(x,\delta'') \subset \cF^{cu}_x(x)$, one of the main difficulties is that, although we have the foliation $\cF^c_S$ within $B_{\delta_0}(\sigma)$, $\cF^c_S$ may not sub-foliate $\cF^{cu}$ since $\cF^c_S$ and $\cF^{cu}$ are given by different dominated splittings (and extended to neighborhoods in different ways). To solve this issue, we blend these two families of fake foliations and define
$$
\bcf^c_x(z) = \cF^{cs}_S(z)\cap\cF^{cu}_x(x), \mbox{ for every } z \in \cF^{cu}_x(x).
$$
$\bcf^c$ is a new $1$-dimensional center fake foliation, which is locally invariant for $f$ since both $\cF^{cs}_S(z)$ and $\cF^{cu}_x(x)$ are invariant, and sub-foliates $\cF^{cu}$.

Next we construct a local product structure around $x_n = f^n(x)$ in the following way. We take $\cD^u$ with $x\in \cD^u \subset \cF^{cu}$ a disk with dimension $\dim E^u_\sigma$, and tangent to the $E^u$ cone. For positive integers $n\le T$, we denote by 
$$
\cD^u_0 = \cD^u \mbox{, and } \cD^u_n \mbox{ the component of } f^n(\cD^u) \cap \cF^{cu}_{x_n}(x_n) \mbox{ containing }x.
$$
Then each $\cD^u_n$ is still tangent to the $E^u$ cone. For each $z\in \cF^{cu}_{x_n}(x_n)$, there is a unique transverse intersection between $\cD^u_n$ and $\bcf^c(z)$, which we write 
$$
[z,x_n] = \bcf^c(z)\pitchfork\cD^u_n.
$$
Note that this local product structure is preserved by $f$, due to the invariance of $\bcf^c$ and the definition of $\cD^u_n$. To be more precise, for each $y\in B_\infty(x,r_0)$ we have 
$$
f^n([y,x]) = [y_n,x_n],\mbox{ for }0\le n \le T.
$$

Recall that $P_z$ is the projection along to flow to $B^\perp_{\delta}(z)$, which is well-defined for $\delta$ small enough. Furthermore, we can make $\delta$ small such that for some $r_\delta>0$,  the time for the projection, which we denoted by $t_z$, satisfies $|t_z|_{B_\delta}| \ll r_\delta\ll r_1$ for every $z \in B_{\delta_0}(\sigma)\setminus J_{\delta_1}(\sigma)$.

To simplify notation, for $y\in B_\infty(x,\delta)$ we take $\tilde{y} = y_{t_0}$ for some $|t_0|<r_\delta$, such that 
$f^T(\tilde{y})=\tilde{y}_T  = P_{x_T}(y_T)$. As a result of Corollary~\ref{c.saturation}, we have 
$$
\tilde{y} \in \cF^{cu}_x(x) \cap B_\infty(x,\delta+r_\delta D_0),
$$ 
where $D_0 = \max\|X\|$ as before.

\begin{figure}[h]
    \centering
    \includegraphics[scale=0.4]{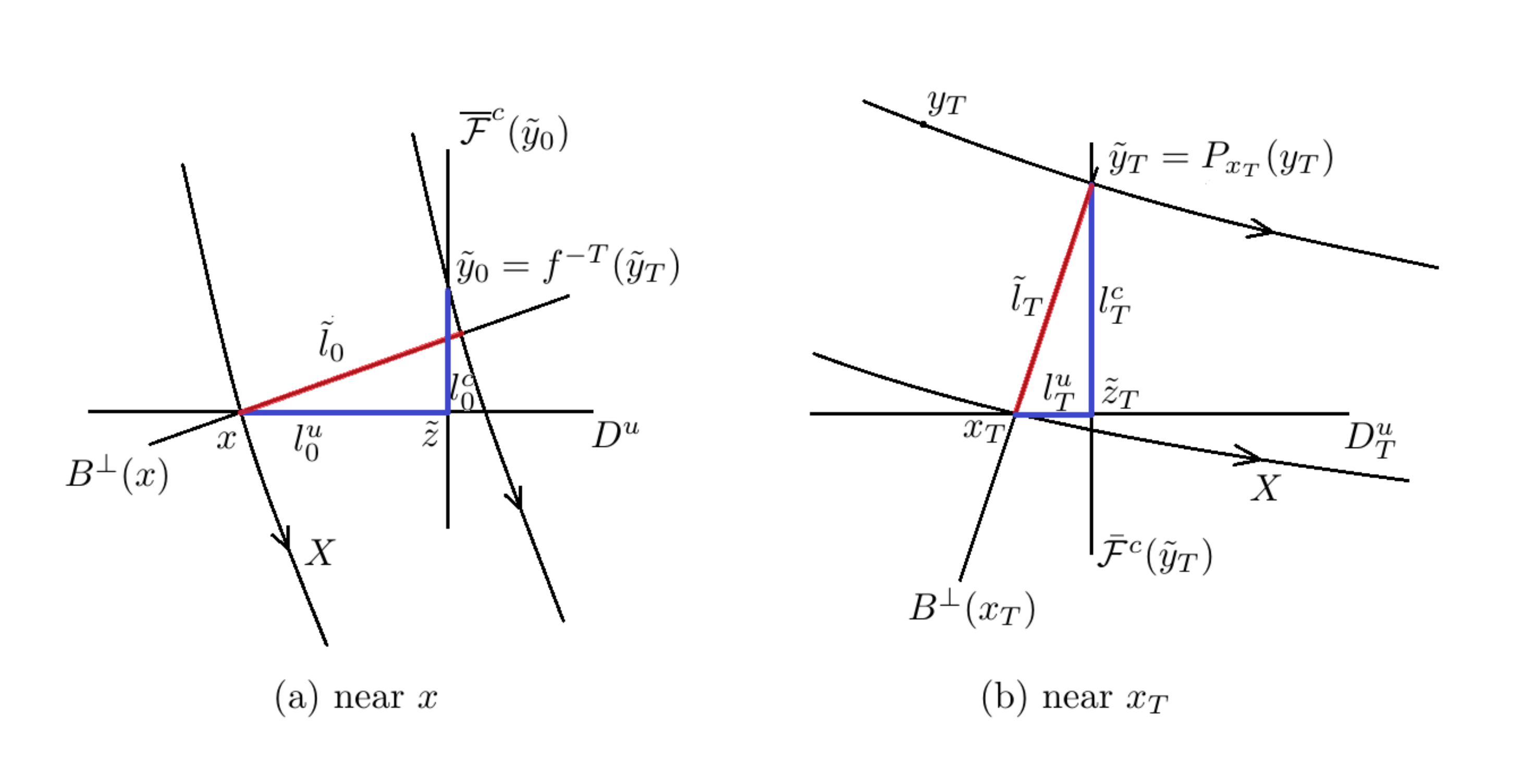}
	\caption{$l^u, l^c$ and $\tilde{l}$.}
	\label{f.3}
\end{figure}
Let $\tilde{z} = [\tilde{y},x]$, then the invariance of $[\cdot,\cdot]$ gives $\tilde{z}_n = f^n(\tilde{z}) = [\tilde{y}_n,x_n]$ for all $n\le T$. By the local product structure,  we take:
\begin{itemize}
\item $l^c_T \subset \bcf^c(\tilde{y}_T) $ joining $\tilde{y}_T$ and $\tilde{z}_T$;
\item $l^u_T \subset \cD^u_T$ the shortest curve (in submanifold metric) connecting $\tilde{z}_T$ and $x_T$.
\end{itemize}
Then $l_T = l^c_T\cup l^u_T$ is a piecewise smooth curve connecting $x_T$ and $\tilde{y}_T$.
Now using the invariance of $\bcf^c$ and the definition of $\cD^u_n$, we can write 
\begin{itemize}
\item $l^c_0 = f^{-T}(l^c_T )\subset \bcf^c(\tilde{y}) $ a curve joining $\tilde{y}$ and $\tilde{z}$;
\item $l^u_0 = f^{-T}(l^u_T) \subset \cD^u_0$ a curve  connecting $\tilde{z}$ and $x$.
\end{itemize}
and $l_0 = l^c_0\cup l^u_0$ is a piecewise smooth curve connecting $x$ and $\tilde{y}$. See Figure~\ref{f.3}.

By the (uniform) transversality between $\bcf^c(\tilde{y}_n)$ and $\cD^u_n$ and the uniform contracting of $f^{-1}$ within $E^u$ cone, there exists $C>1$ and $\lambda_1>1$ such that
\begin{equation}\label{e.length.T}
\max\{\length(l^u_T),\length(l^c_T) \} \le Cd(x_T,\tilde{y}_T),
\end{equation}
and 
\begin{equation}\label{e.length}
\length(l^c_0)\le C(\delta+r_\delta D_0), \mbox{ and } \length(l^u_0)\le \lambda_1^{-T}\length(l^u_T).
\end{equation}

Now we shrink $\delta$ one last time, denote by $\delta''$, such that:
\begin{enumerate}
\item $P_x(l_0) = \tilde{l}_0$ is well-defined, with $|t_x|_{l_0}|\le t_0\ll r_1$, where $t_0$ is small enough such that $\frac{1}{C}\le |\Jac(\Phi_t|_{\Sigma})|\le C$ for $0\le t<t_0$ and any two-dimensional subspace $\Sigma\subset T_xM$.
\item $\delta$ satisfies Lemma~\ref{l.comparedistance}:
$$
d(x,y)\le d^*_x(x,y) \le Cd(x,y),
$$
and the same holds for $x_T$ and $y_T$.
\end{enumerate}

Let $\tilde{l}^u_0 = P_x(l^u_0)$ and  $\tilde{l}^c_0 = P_x(l^c_0)$. Then we have 
\begin{equation}\label{e.length.0}
d^*_x(x,P_x(y)) \le Cd(x,P_x(y)) \le C(\length(\tilde{l}^u_0) + \length(\tilde{l}^c_0)).
\end{equation}
It remains to estimate the  length of $\tilde{l}^u_0$ and  $\tilde{l}^c_0$.

Note that $l^u_0$ is tangent to the $E^u$ cone, and thus transverse to the flow direction (which is tangent to the $E^s \oplus E^c$ cone if $\varepsilon$ is small). On the other hand, $\tilde{l}^u_0 $ is perpendicular to the flow directions. Since  both $\tilde{l}^u_0$ and $l^u_0$ are transverse to the flow directions, there is $C_1>0$ such that  
$$
\frac{1}{C_1} \le \frac{\length(\tilde{l}^u_0)}{\length(l^u_0)}\le C_1.
$$
This together with (\ref{e.length}) shows that 
\begin{equation}\label{e.length.u}
\length(\tilde{l}^u_0) \le C_2\lambda_1^{-T}\length(l^u_T).
\end{equation}

Next we estimate $\length(\tilde{l}^c_0)$ using the sectional hyperbolicity.

Let $H$ be the holonomy map from $\tilde{l}^c_T: =P_{x_T}(l^c_T)$ to $\tilde{l}^c_0$ induced by the flow. Similar to the $l_0^u$ case, $\tilde{l}^c_T$ is perpendicular to the flow direction, while $l^c_T$ is tangent to the $E^c$ cone (well-defined in a neighborhood of $\sigma$) but the flow direction is tangent to the $E^u$ cone for $\varepsilon$ small.  This shows that  
$$
\frac{1}{C_1} \le \frac{\length(\tilde{l}^c_T)}{\length(l^c_T)}\le C_1.
$$
Recall that $\cP[u,v]$ is the parallelogram generated by vectors $u$ and $v$.
If we take $v$ a tangent vector of $\tilde{l}^c_T$ at $z\in \tilde{l}^c_T$ and consider $H\cP[v,X(z)] = \cP[DH^{-1}v, X(H^{-1}(z))]$, sectional hyperbolicity of $f$ implies that 
$$
\|DH^{-1}v\|\|X(H^{-1}(z))\| \le C_3\lambda_1^{-T}\|v\|\|X(z)\|,
$$
which shows that
\begin{equation}
\length(\tilde{l}^c_0) \le C_3'\lambda_1^{-T}\frac{\|X(x_T)\|}{\|X(x)\|}\length(l^c_T).
\end{equation}
Combine this with (\ref{e.length.T}), (\ref{e.length.0}) and (\ref{e.length.u}), we have
\begin{align*}
d^*_x(P_x(y)) \le&  C(\length(\tilde{l}^u_0) + \length(\tilde{l}^c_0))\\
\le & C(C_2\lambda_1^{-T}\length(l^u_T)+C_3'\lambda_1^{-T}\frac{\|X(x_T)\|}{\|X(x)\|}\length(l^c_T))\\
\le & C_4\lambda_1^{-T}(\frac{D_0}{\delta_1}+1)d(x_T,\tilde{y}_T)\\
\le & C_4\lambda_1^{-T}(\frac{D_0}{\delta_1}+1) d^*_{x_T}(x_T,P_{x_T}(y_T)).
\end{align*}
Since $T\to\infty$ as $\varepsilon\to 0$, for any given $L>1$, we can take $\varepsilon$ small enough such that $C_4\lambda_1^{-T}(\frac{D_0}{\delta_1}+1) < \frac1L$. This finishes the proof of Lemma~\ref{l.expansion.near}.
\end{proof}

\subsection{Positive entropy}\label{ss.positiveentropy} 
In this section we will prove Theorem~\ref{m.positiveentropy}. The proof consists of two steps. First we prove that in every small neighborhood $U$ of $\Lambda$, the topological entropy of $f|_U$ is bounded from below by the volume expanding rate along $F^{cu}$ bundle. Then we take a sequence of such neighborhoods shrinking to $\Lambda$, and use the upper semi-continuity of the metric entropy to obtain an lower bound for $h_{top}(f|_\Lambda)$.

First we introduce the volume expansion rate on a bundle.
\begin{definition}
Let $D$ be a disk tangent to  the $F$ cone. The {\em volume expansion
of $D$}, which we denote by  $v_F(D)$, is defined by
$$
\limsup_n\frac1n\log(\Leb(g^n(D))).
$$
The {\em volume expansion $v_F$ of bundle $F$} is defined by:
$$
v_F = \sup\{v_F(D): D\mbox{ is tangent to the } F \mbox{ cone}\}.
$$
\end{definition}

The positivity for the topological entropy of $f|_U$ relies on the following theorem:
\begin{theorem}\label{t.positiveentropy}
Suppose $g$ is a diffeomorphism which admits a dominated splitting $E \oplus F$. Then $h_{top}(g) \ge v_F$.
\end{theorem}

Theorem~\ref{t.positiveentropy} was first stated in~\cite{LVY}, where the bundle $F$ is required to be  uniformly expanding. The main reason is that, in the proof of \cite[Proposition 2]{LVY}, they need $f^n(D \cap B_n(x, \varepsilon))$ to be `almost' a ball, which occurs only when the bundle $F$ is uniformly expanding. In general, this set may not even be connected.
\begin{proof}[Proof of Theorem~\ref{t.positiveentropy}]
For any $\delta > 0$, let $D$ be a disk which is tangent to the $F$ cone and satisfies 
$v_F (D) \ge v_F - \delta$, that is:
$$
\limsup_n\frac1n\log(\Leb(g^n(D))) \ge v_F - \delta.
$$
Take $\varepsilon > 0$ and 
$\Gamma_n = \{x_1, \ldots , x_{r_n(D,\varepsilon)}\}$
an $(n, \varepsilon)$-spanning set of $D$. 
By Lemma~\ref{l.boundedvolumebowen},
$$ 
\Leb(g^n(D \cap B_n(x, \varepsilon))) \le K.
$$
Since that $\{B_n(x, \varepsilon)\}_{x\in\Gamma_n}$ is a cover of $D$, we get
$$
\Leb(f^n(D)) \le \sum_{i=1}^{r_n(D,\varepsilon)}\Leb(f^n(D \cap B_n(x_i, \varepsilon))) \le Kr_n(D, \varepsilon).
$$
Taking logarithm and divide by $n$, we obtain

\begin{align*}
\limsup_n\frac1n\log(r_n(D, \varepsilon)) &\ge \limsup\frac1n\log(\Leb(f^n(D))/K)\\&= \limsup \frac1n\log(\Leb f^n(D)) \\&\ge v_F - \delta,
\end{align*}
which implies that $r(D,\varepsilon) \ge v_F - \delta$. Since $\delta>0$ is arbitrary, we
conclude the proof of Theorem~\ref{t.positiveentropy}.
\end{proof}
\begin{remark}
Similar to Lemmas~\ref{l.boundedvolume}, \ref{l.invariantcone} and \ref{l.boundedvolumebowen}, if $g$ has a dominated splitting $E\oplus F$ on an invariant set $\Lambda$, one can extend the invariant cone field $C_a(F)$ to a  neighborhood of $\Lambda$, such that the extension is still forward invariant. It is easy to check that the proof of Theorem~\ref{t.positiveentropy} applies to points in the neighborhood.
\end{remark}

Now we are ready to show that every Lorenz-like class has positive topological entropy.
\begin{proof}[Proof of Theorem~\ref{m.positiveentropy}]
Recall that $f$ is the time-one map of a $C^1$ flow $\phi_t$, and $\Lambda$ is a compact invariant set of $\phi_t$ that is sectional hyperbolic.  We take a forward invariant $F^{cu}$ cone on $\Lambda$ along the bundle $F^{cu}$ which is assumed to be sectional-expanding.

Let $\{U_i\}^\infty_{i=1}$ be a sequence of neighborhoods of $\Lambda$ as in Definition~\ref{d.lorenzclass}, i.e., they
satisfy:
\begin{itemize}
\item $U_1 \supset U_2\supset\cdots$ and $\bigcap_iU_i = \Lambda$;
\item for each $i \ge 1$, $\phi_t(U_{i+1}) \subset U_i$ for any $t \ge 0$.
\end{itemize}
Assume $U_1$ is taken small enough, such that the $F^{cu}$ cone can be extended to $U_1$ (thus to $U_i$ for every $i\ge 1$) and is still forward invariant. Then for any disk $D \subset U_1 $which is tangent to the $F$ cone, $f^n(D)$ is tangent to the $F$ cone. 

Since $F^{cu}|_\Lambda$ is volume expanding (which also holds in $U$), there exist $C > 0$ and $ \lambda > 0$ such that
$$
\Leb(f^n(D)) \ge Ce^{\lambda n}.
$$
This implies $v_F(D) \ge \lambda$. Theorem~\ref{t.positiveentropy} applied to $f|_{U_i}$ yields
$$
h_{top}(f|_{U_i}) \ge \lambda.
$$
By the variational principle, there are $f$ ergodic invariant measures $\mu_i$ supported within $U_i$, with 
$$
h_{\mu_i} \ge \lambda -\frac1i.
$$
Passing to a subsequence, we get $\mu_i \to \mu_0$ in the weak-* topology, where $\mu_0$ must be supported on $\Lambda$. By Theorem~\ref{m.h-expansive}, $f|_{U_i}$ is entropy expansive, thus the metric entropy is upper semi-continuous. This shows that 
$$
h_{\mu_0}\ge \lambda>0.
$$
We conclude that
$$
h_{top}(\phi_t|_\Lambda) = h_{top}(f|_\Lambda) \ge h_{\mu_0} \ge \lambda>0,
$$
the proof is complete.
\end{proof}

\section{Continuity of topological entropy}\label{s.continuity}
In this section we will prove Theorem~\ref{m.continuity} by showing that the support of every hyperbolic measure can be approximated by horseshoes with large entropy.

Let $\Lambda$ be a sectional hyperbolic compact invariant set for a $C^1$ flow $\phi_t$, and $U$ be a neighborhood of $\Lambda$. Denote by $\tl_n$ and $\tl$ the maximal invariant set of $\phi^n_t$ and $\phi_t$ in $U$, respectively. 
If $h_{top}(\phi_t|_{\tl}) = 0$ (recall that in order to get positive entropy, we need $\Lambda$ to be Lyapunov stable), then $\lim_{n\to\infty}h_{top}(\phi^n_t|_{\tl_n}) \ge 0 = h_{top}(\phi_t|_{\tl})$ holds trivially. Therefore, we can assume that $$h_{top}(\phi_t|_{\tl}) >0. $$

Note that by Theorem~\ref{m.h-expansive}, there is a neighborhood $U$ of $\Lambda$, such that $\phi_t$ is entropy expansive in $U$. If we denote by $f = \phi_1$ the time-one map of $\phi_t$ as before, then $f$ is entropy expansive  and satisfies $h_{top}(f|_{\tl})>0$. We can apply Lemma~\ref{l.variationalprinciple} to $\phi_t$ and obtain a measure of maximal entropy, which we denote by $\mu_{\phi_t}$. Since $\supp\mu_{\phi_t} \in U$, it follows that $\phi_t$ on $\supp\mu_{\phi_t}$ is sectional hyperbolic. In particular, $\mu_{\phi_t}$ is a hyperbolic measure.

The next theorem shows that every hyperbolic measure $\mu$ of $\phi_t$ with positive entropy can be approximated by horseshoes with entropy close to  $h_{top}(\phi_t|_{\supp\mu})$.

\begin{theorem}\label{t.horseshoe}
Let $\mu$ be a hyperbolic, ergodic measure of $C^1$ flow $\phi_t$ with positive entropy. Assume that there is a dominated splitting $E\oplus F$ on $\supp \mu$, such that $\dim E$ is the (stable) index of $\mu$. Then for every $\varepsilon>0$, there is a hyperbolic set $\Lambda_\varepsilon$ in a small neighborhood of $\supp\mu$, uniformly away from singularities and containing some periodic orbit, with 
$h_{top}(\phi_t|_{\Lambda_\varepsilon}) > h_{top}(\phi_t|_{\supp\mu}) -\varepsilon$.
\end{theorem}
These type of approximation results are well known in the $C^{1+}$ context, see for example~\cite{LS}.

For $\varepsilon$ small enough, the hyperbolic set  given by the above theorem must be contained in $U$ and outside a neighborhood (with uniform size) of $\Sing(X)$. Since the topological entropy of a hyperbolic set varies continuously, it follows that $h_{top}(\cdot|_{\tl})$ is lower semi-continuous at $\phi_t$. The upper semi-continuity follows from Theorem~\ref{m.h-expansive} and Lemma~\ref{l.uppersemicontinuity}.
When $\Lambda$ is a chain recurrent class, Lemma~\ref{l.shadowing}(f) below shows that $\Lambda$ contains a periodic orbit.
This concludes the proof of Theorem~\ref{m.continuity}, leaving only the proof of Theorem~\ref{t.horseshoe}. 

The proof of Theorem~\ref{t.horseshoe} uses a similar argument of Katok in~\cite{K80} for diffeomorphisms. Note however, that the original argument of~\cite{K80} cannot be directly applied to flows even if the flow is uniformly hyperbolic without singularities. The main obstruction is due to the shadowing lemma for flows only allows one to compare the pseudo-orbit and the shadowing orbit up to a change of time. We overcome this issue using a shadowing lemma by Liao~\cite{Liao}. See Lemma~\ref{l.shadowing} below, in particular item (d).

We organize this section in the following way. In~\ref{ss.pesin} we establish the scaled linear Poincar\'e flow and its expanding property. In~\ref{ss.shadowing} we will introduce Liao's shadowing lemma, which allows us to shadow pseudo-orbits that pass through neighborhoods of singularities and estimates the time difference between the pseudo-orbit and the shadowing orbit. Finally, we will prove Theorem~\ref{t.horseshoe} in Section~\ref{ss.shadowing}.

\subsection{Scaled linear Poincar\'e flows}\label{ss.pesin}
Starting from now, $\mu$ will be a non-trivial hyperbolic ergodic measure with a dominated splitting $E\oplus F$ on  $\supp\mu$, such that $\dim E$ is the index of $\mu$. 

Recall that for a regular point $x$ and $v\in T_xM$, the linear Poincar\'e flow $\psi_t: N_x\to N_{\phi_t(x)}$ is the projection of $\Phi_t(v)$ to $N_{\phi_t(x)}$, where $N_x$ is the orthogonal complement of $X(x)$. The {\em scaled linear Poincar\'e flow}, which we denote by $\psi^*_t$, is defined as
\begin{equation}
\psi^*_t(v) = \frac{\|X(x)\|}{\|X(\phi_t(x))\|}\psi_t(v) = \frac{\psi_t(v)}{\|\Phi_t|_{<X(x)>}\|}.
\end{equation}  

\begin{lemma}\label{l.scaledflow}
$\psi^*_t$ is a bounded cocycle over $N_\Lambda$ in the following sense: for any $\tau>0$, there is $C_\tau>0$ such that for any $t\in[-\tau,\tau]$,
$$
\|\psi^*_t\|\le C_\tau.
$$
Furthermore, for every non-trivial ergodic measure $\mu$, the cocycles $\psi_t$ and $\psi^*_t$ have the same Lyapunov exponents and Oseledets splitting.
\end{lemma}

\begin{proof}
Note that for each $\tau>0$ and $t\in [-\tau,\tau]$, $\|\psi_t\|$ is bounded from above and $\|\Phi_t\|$ is bounded away from zero. So the upper bound of $\psi_t^*$ follows. 

For the Lyapunov exponents, we have 
$$
\lim_{t\to\pm\infty}\frac1t \log \|\psi_t^*{(v)}\| = \lim_{t\to\pm\infty}\frac1t (\log \|\psi_t(v)\| -\log \|\Phi_t|_{<X(x)>}\|),
$$
but $\limsup\frac1t\log \|\Phi_t|_{<X(x)>}\| \le 0$ since the flow speed is bounded.  The $\liminf$ follows by considering $-X$.
\end{proof}

Recall that $\pi:T_xM \to N_x$ is the orthogonal projection along flow direction. 
\begin{lemma}\label{l.domination.scaled}
We have $X|_{\supp\mu\setminus\Sing(X)}\subset F$. Furthermore, $\pi(E)\oplus \pi(F)$ is also a dominated splitting on $N_{\Lambda\setminus\Sing(X)}$ for both $\psi_t$ and $\psi^*_t$, which is also the Oseledets splitting for $(\psi_t,\mu)$ corresponding to the negative exponents and positive exponents. 
\end{lemma}

\begin{proof}
Since $\mu$ is hyperbolic with index $\dim E$, $f=\phi_1$ has precisely $\dim E$ many negative exponents, and a vanishing exponent given by the flow direction. Since $E\oplus F$ is dominated, Lyapunov exponents on $F$ must be larger than those in $E$, thus non-negative. It then follows that $E$ is the Oseledets splitting corresponding to the negative exponents, and  $X|_{\supp\mu\setminus\Sing(X)}\subset F$.

Next we will show the second part of this lemma only for $\psi_t$. The result for $\psi^*_t$ will then follow from Lemma~\ref{l.scaledflow}.

Take any $x\in\supp\mu$ and unit vectors $u\in\pi(E_x)$, $v\in\pi(F_x)$. Since $X|_{\supp\mu\setminus\Sing(X)}\subset F$, we have $v\in\pi(F_x)\subset F_x$. Let $u'=\pi^{-1}u\in E_x$ and $v'=\Phi_{-t}\circ\psi_t(v)\in F_x$. Since $E\oplus F$ is dominated for $\phi_t$, we must have 
$$
\frac{\|\Phi_t(u')\|}{\|\Phi_t(v')\|} < C \lambda^t\frac{\|u'\|}{\|v'\|},
$$ 
for some $C>0$ and $\lambda\in(0,1)$. 

Since $E\oplus F$ is dominated and  $X|_{\supp\mu\setminus\Sing(X)}\subset F$, the angle between $E$ and $X$ must be away from zero. On the other hand, since  $N_x$ is the orthogonal complement of $X$, the angle between  $E$ and $N_x$ must be away from $\pi/2$. This shows that there exists some constant $C_1$ independent of $x$, such that for all unit vectors $u\in \pi(E_x)$, $\|u'\|=\|\pi^{-1}u\|<C_1$. Note also that $\Phi_t(v') = \psi_t(v)$. It follows that 
\begin{equation}\label{e.scale1}
\frac{\|\Phi_t(u')\|}{\|\psi_t(v)\|} < CC_1 \lambda^t\frac{1}{\|v'\|}.
\end{equation}

From the definition of $\psi_t$ and $\pi$, we have 
$$
\psi_t(u) = \pi\circ\Phi_t(u) = \pi\circ\Phi_t(u') \,\mbox{ and }\pi(v')=v,
$$
thus
\begin{equation}\label{e.scale2}
\|\psi_t(u)\|\le \|\Phi_t(u')\| \mbox{ and } \|v'\|\ge \|v\|=1,
\end{equation}

Combining (\ref{e.scale1}) and (\ref{e.scale2}) we get
$$
\frac{\|\psi_t(u)\|}{\|\psi_t(v)\|}< CC_1 \lambda^t.
$$
Finally,  using the fact that $\pi(v')=v$ we get $\|v'\|\ge \|v\|=1$ and therefore
$$
\frac{\|\psi_t(u)\|}{\|\psi_t(v)\|}< CC_1 \lambda^t.
$$
This shows that $\pi(E)\oplus\pi(F)$ is a dominated splitting for $\psi_t$.

Since $\mu$ has index $\dim E = \dim (\pi(E))$, the argument used at the beginning of this proof shows that $\pi(E)\oplus\pi(F)$ is indeed the Oseledets splitting for $\psi_t$ corresponding to negative and positive exponents.
\end{proof}

Next we describe the hyperbolicity for the scaled linear Poincar\'e flow $\psi^*_t$.

\begin{definition}\label{d11}
For $T_0>0, \lambda\in(0,1)$, the orbit segment $\{\phi_t(x)\}_{[0,T]}$ is called {\em $(\lambda, T_0)^*$ quasi-hyperbolic} with respect to a splitting $N_x = E^N_x\oplus F^N_x$ and the scaled linear Poincar\'e flow $\psi^*_t$, if there exists a partition
$$
0=t_0<t_1<\cdots<t_l=T, \mbox{\hspace{.2cm} where } t_{i+1}-t_i\in [T_0,2T_0], 
$$
such that for $k=1,\ldots,l-1$, we have
$$
\prod_{i=0}^{k-1}\|\psi^*_{t_{i+1} - t_i}|_{\psi_{t_i}(E^N_x)}\| \le \lambda^k; \quad \prod_{i=k}^{l-1}m(\psi^*_{t_{i+1} - t_i}|_{\psi_{t_i}(F^N_x)}) \ge \lambda^{-(l-k)},
$$
and 
$$
\frac{\|\psi^*_{t_{i+1} - t_i}|_{\psi_{t_i}(E^N_x)}\|}{m(\psi^*_{t_{i+1} - t_i}|_{\psi_{t_i}(F^N_x)})} \le \lambda^2.
$$
\end{definition}

\begin{definition}\label{d12}
For $T_0>0$, $\lambda\in (0,1)$, an orbit segment $\phi_{[0,T]}(x)$ is called $(\lambda,T_0)$-forward contracting for the bundle $E\subset N_x$, if there exists a partition
$$
0=t_0<t_1<\cdots<t_n=T, \mbox{\hspace{.2cm} where } t_{i+1}-t_i\in [T_0,2T_0], 
$$
such that for all $k=1,\ldots,n-1$,
\begin{equation}\label{e.hyptime}
\prod_{i=0}^{k-1}\|\psi^*_{t_{i+1} - t_i}|_{\psi_{t_i}(E)}\| \le \lambda^k.
\end{equation}
Similarly, an orbit segment $\phi_{[-T,0]}(x)$ is called $(\lambda,T_0)$-backward contracting for the bundle $E\subset N_x$, if it is forward contracting for the flow $-X$.

A point $x$ is called a {\em $(\lambda,T_0)$-forward hyperbolic time for the bundle $E\subset N_x$}, if the infinite orbit $\phi_{[0,+\infty)}$ is $(\lambda,T_0)$-forward contracting. In this case the partition is taken as
$$
0=t_0<t_1<\cdots<t_n<\ldots, \mbox{\hspace{.2cm} where } t_{i+1}-t_i\in [T_0,2T_0], 
$$
and~\eqref{e.hyptime} is stated for all $k\in\NN$. Similarly,
$x$ is called a {\em $(\lambda,T_0)$-backward hyperbolic time for the bundle $E\subset N_x$}, if it is a forward hyperbolic time for $-X$.  $x$ is called a {\em two-sided hyperbolic time}, if it is both a  forward and backward hyperbolic time.

\end{definition}
By the classic work of Liao~\cite{Liao}, there exists $\delta>0$ such that if $x$ is a backward hyperbolic time, then $x$ has unstable manifold with size $\delta \|X(x)\|$. Similarly, if $x$ is a forward hyperbolic time then it has stable manifold with size $\delta \|X(x)\|$.  In both cases, we say that $x$ has unstable/stable manifold {\em up to the flow speed}.

The next lemma can be seen as a $C^1$ version of the Pesin theory for flows.
\begin{lemma}\label{l.C1pesin}
For almost every ergodic component $\tilde{\mu}$ of $\mu$ with respect to $f=\phi_1$, there are $L',\eta,T_0>0$ and a compact set $\Lambda_0\subset\supp\mu\setminus\Sing(X)$ with  positive $\tilde{\mu}$ measure, such that for every $x$ satisfying $f^n(x)\in \Lambda_0$ for  $n>L'$, the orbit segment $\{\phi_t(x)\}_{[0,n]}$ is $(\eta,T_0)^*$ quasi-hyperbolic with respect to the splitting $N_x = \pi(E_x)\oplus \pi(F_x)$ and the scaled linear Poincar\'e flow $\psi^*_t$.
\end{lemma}

\begin{proof}
The proof is very standard. 
By Lemma~\ref{l.domination.scaled}, for $\tilde{\mu}$ almost every $x$, $\pi(E_x)\oplus \pi(F_x)$ is the Oseledets splitting of $\psi^*_t$ corresponding to the negative and positive exponents. By the subadditive ergodic theorem, there is $a<0$ such that for $N_0$ large enough, we have 
$$
\frac{1}{N_0} \int\log\|\psi^*_{N_0}|_{\pi(E)}\|\,d\tilde{\mu} < a \,\,\mbox{ and }\,\lim_{t\to -\infty}\frac{1}{N_0} \int\log\|\psi^*_{-N_0}|_{\pi(F)}\|\,d\tilde{\mu} < a.
$$
Let 
$$
\tilde{\mu} = \frac{1}{k_0}(\tilde{\mu}_1+\cdots\tilde{\mu}_{k_0})
$$
be the ergodic decomposition of $\tilde{\mu}$ with respect to $f^{N_0}$. Change the order if necessary, we may assume that 
$$
\frac{1}{N_0} \int\log\|\psi^*_{N_0}|_{\pi(E)}\|\,d\tilde{\mu}_1 < a \,\,\mbox{ and }\,\frac{1}{N_0} \int\log\|\psi^*_{-N_0}|_{\pi(F)}\|\,d\tilde{\mu}_1 < a.
$$
By the Birkhoff ergodic theorem on $f^{N_0}$, for $\tilde{\mu}_1$ almost every $x$, 
$$
\lim_{m\to\infty} \frac{1}{mN_0}\sum\limits_{i=0}^{m-1}\log\|\psi^*_{N_0}|_{{\pi(E_{f^{iN_0}(x)})}}\| < a,
$$
and similarly on $\pi(F)$:
$$
\lim_{m\to\infty} \frac{1}{mN_0}\sum\limits_{i=0}^{m-1}\log\|\psi^*_{-N_0}|_{{\pi(F_{f^{-iN_0}(x)})}}\| < a.
$$
Take $n_x>0$ such that the above inequalities holds for all $m>n_x$, and $N_1$ such that the set $\Lambda' = \{x: n_x<N_1\}$ has positive $\tilde{\mu}_1$ measure. Let $\Lambda_0\subset\Lambda'\setminus\Sing(X)$ be compact and has positive $\tilde{\mu}_1$ measure. Then $\tilde{\mu}(\Lambda_0)>0$. By Lemma~\ref{l.scaledflow}, we can take
$$
K = \max_{|t|<N_0, y\in\supp\mu\setminus\Sing(X)}\left\{\sup\{\psi^*_t|_{E_y}\};\sup\{\psi^*_t|_{F_y}\}\right\}.
$$
Choose $N_2$ large enough such that 
$$
\frac{N_2+N_0}{N_0}a + 3K < b <0
$$
for some $b<0$. We claim that for any sequence $n_1<n_2<\ldots<n_l$ with $N_2\le n_{i+1}-n_i\le N_2+N_0$ for each $i$, we have 
$$
\frac{1}{l}\sum_{i=0}^{l-1}\log \|\psi^*_{n_{i+1}-n_i}|_{\pi(E_{f^{n_i}(x)})}\| < b <0,
$$
and a similar inequality holds on $\pi(F)$. The lemma will then follow from this claim and the domination between $\pi(E)$ and $\pi(F)$. 

To prove this claim, for each $i=1,\ldots,l-1$, write
$$
k_i = [\frac{n_{i+1}}{N_0}] - [\frac{n_i}{N_0}] - 1, \, n_i' = ([\frac{n_i}{N_0}]+1)N_0 \mbox{ and } n^*_{i+1} = [\frac{n_{i+1}}{N_0}]N_0.
$$
Then we have $n^*_i \le n_i \le n'_i$, $n^*_{i+1}-n_i'=k_iN_0$, and
$$
\psi^*_{n_{i+1}-n_i}|_{\pi(E_{f^{n_i}(x)})} = \psi^*_{n_{i+1}-n^*_{i+1}}|_{\pi(E_{f^{n^*_{i+1}}(x)})}\circ\psi^*_{k_iN_0}|_{\pi(E_{f^{n'_i}(x)})}\circ\psi^*_{n'_{i}-n_{i}}|_{\pi(E_{f^{n_{i}}(x)})}
$$
Note that $n_{i+1}-n^*_{i+1} \le N_0$ and $n'_{i}-n_{i} \le N_0$. By the choice of $K$, we have 
\begin{align*}
\log \|\psi^*_{n_{i+1}-n_i}|_{\pi(E_{f^{n_i}(x)})}\| \le& 2K+ \log \|\psi^*_{k_iN_0}|_{\pi(E_{f^{n'_i}(x)})}\|\\
\le& 3K+\log \|\psi^*_{n'_{i+1}- n'_i}|_{\pi(E_{f^{n'_i}(x)})}\|.
\end{align*}
Sum over $i$, we obtain
\begin{align*}
\frac{1}{l}\sum_{i=0}^{l-1}\log \|\psi^*_{n_{i+1}-n_i}|_{\pi(E_{f^{n_i}(x)})}\| \le&\, \frac1l\sum_{i=0}^{l}\log \|\psi^*_{n'_{i+1}- n'_i}|_{\pi(E_{f^{n'_i}(x)})}\| +3K\\
\le& \,\frac1l\sum_{i=0}^{\frac{n'_{l+1}}{N_0}}\log\|\psi^*_{N_0}|_{\pi(E_{f^{jN_0}(x)})}\| +3K\\
\le& \, \frac{n'_{l+1}}{l\, N_0}a + 3K \le \frac{N_2+N_0}{N_0}a + 3K <b <0.
\end{align*}
\end{proof}


\subsection{A shadowing lemma by Liao and proof of Theorem~\ref{t.horseshoe}}\label{ss.shadowing}
In this section we will introduce a shadowing lemma by Liao~\cite{Liao} for the scaled linear Poincar\'e flow. 

\begin{lemma}\label{l.shadowing}
Given a compact set $\Lambda_0$ with $\Lambda_0 \cap \Sing(X) = \emptyset$ and $\eta\in(0,1),T_0>0$, for any $\varepsilon>0$ there exists $\delta>0$, $L>0$ and $\delta_0>0$, such that for any $(\eta,T_0)^*$ quasi-hyperbolic orbit segment $\{\phi_t(x)\}_{[0,T]}$ with respect to a dominated splitting $N_x = E_x \oplus F_x$ and the scaled linear Poincar\'e flow $\psi^*_t$, if $x,\phi_T(x) \in \Lambda_0$ with $d(x,\phi_T(x))<\delta$, then there exists a point $p$ and a $C^1$ strictly increasing function $\theta:[0,T] \to \RR$, such that
\begin{enumerate}[label=(\alph*)]
\item $\theta(0)=0$ and $|\theta'(t)-1|<\varepsilon$;
\item $p$ is a periodic point with $\phi_{\theta(T)}(p)=p$;
\item $d(\phi_t(x), \phi_{\theta(t)}(p))\le \varepsilon\|X(\phi_t(x))\|$, for all $t\in[0,T]$;
\item $d(\phi_t(x), \phi_{\theta(t)}(p))\le Ld(x,\phi_{T}(x))$;
\item $p$ has stable and unstable manifold with size at least $\delta_0$.
\item if $\Lambda_0\subset\Lambda$ for a sectional hyperbolic chain recurrent class $\Lambda$, then $p\in\Lambda$.
\end{enumerate}
Furthermore, the result remains true with the same constants $\delta$, $L$ and $\delta_0>0$ if $\Lambda_0$ is replaced by a subset of  $\Lambda_0$.
\end{lemma}  

\begin{remark}
By (a), the period of the shadowing orbit, $\theta(T)$, satisfies $\theta(T)\in[T(1-\varepsilon), T(1+\varepsilon)]$. However, using the fact that 
$$
\sum\limits_{i=0}^{l-1} d(\phi_{\theta(t_i)}(p),\phi_{t_{i}}(x)) \le L^* d(x,\phi_{T}(x)).
$$ 
(note that the constant $L^*$ on the right hand does not depend on $T$), one can show the following modified version of (a):
\begin{enumerate}[label=(\alph*')]
\item we have $\theta(0)=0$ and $|\theta'(t)-1|<\varepsilon$; furthermore, there is a constant $C$ independent of $T$, such that $\theta(T)\in [T-C\varepsilon, T+C\varepsilon]$.
\end{enumerate}
\end{remark}

Now we are ready to prove Theorem~\ref{t.horseshoe}.

Let $\tilde{\mu}$ be a typical ergodic component of $\mu$ with respect to $f$. By Lemma~\ref{l.componententropy}, $h_{\tilde{\mu}}(f) = h_\mu(f)>0$. Let $\Lambda_0$ be the compact set with positive $\tilde{\mu}$ measure given by Lemma~\ref{l.C1pesin}.  Also let $L',T_0,\eta>0$ be the constants given by the same lemma. Apply Lemma~\ref{l.shadowing} with $\Lambda_0$ and $\eta, T_0$, for every $\varepsilon > 0$ we obtain $\delta,L$ and $\delta_0$. 

Replace $\Lambda_0$ by a compact subset if necessary, we may assume that $\Lambda_0$ is away from singularities with diameter small enough, such that any two periodic points obtained by Lemma~\ref{l.shadowing} are homoclinic related. Following the proof of~\cite[Theorem 4.3]{K80}, for every $\alpha,l>0$ and $n\in \NN$, there is a finite set $K_n = K_n(\alpha,l)$ with the following property:
\begin{itemize}
\item $K_n \subset \Lambda_0$;
\item for $x,y\in K_n$, $d_n^f(x,y) = \max_{0\le j\le n-1}\{d(f^jx,f^jy)\} > \frac1l$;
\item for every $x\in K_n$, there is an integer $m(x)$ with $n\le m(x)\le (1+\alpha)n$, such that $f^{m(x)}(x)\in \Lambda_0$ with $d(x,f^{m(x)}(x)) < \frac{1}{4Ll}$;
\item $\lim_{l\to\infty}\lim_{n\to\infty}\frac{1}{n}\log\Card K_n(\alpha, l) \ge h_{\tilde{\mu}}(f)-\alpha$, where $\Card K$ for a subset $K$ denotes the cardinality of $K$.
\end{itemize}
We take $n,l$ large enough, such that $n>L'$, $\frac{1}{4l}<\delta$ and 
$$
\Card K_n(\alpha,l) > \exp(n(h_{\tilde{\mu}} - \alpha)).
$$

For every $x\in K_n$, by Lemma~\ref{l.shadowing} the orbit segment $\{\phi_t(x)\}_{[0,m(x)]}$ is shadowed by a periodic point $p_x$ with period no more than $n(1+\varepsilon)(1+\alpha)$. Item (d) in Lemma~\ref{l.shadowing} guarantees that $d^f_n(x,p_x) \le L d(x,f^{m(x)}(x)) < \frac{1}{4l}.$ As a result,
$$
d^f_n(p_x,p_y) \ge d^f_n(x,y) - d^f_n(x,p_x) -  d^f_n(y,p_y)  > \frac{1}{2l}.
$$
Note that different $x,y \in K_n$ may be shadowed by the same periodic orbit $\Orb(p)$. When this happens, we must have $p_x = \phi_{t_{y,x}}(p_y)$ for some $t_{y,x}$. Then the estimate above implies that 
$$
t_{y,x} > \frac{1}{2lD},
$$
where $D = \max\{\|X\|\}$ is the maximum of the flow speed as before. Therefore, for each $x\in K_n$, the periodic orbit $\Orb(p_x)$ can shadow no more than $2lDn(1+\varepsilon)(1+\alpha)$ different points in $K_n$.

As a result, there are at least 
$$
k_n = \frac{\exp(n(h_{\tilde{\mu}} - \alpha))}{2lDn(1+\varepsilon)(1+\alpha)}
$$
many different periodic orbits, with periodic at most $n(1+\varepsilon)(1+\alpha)$. Since they are homoclinic related near $\Lambda_0$, we have a horse-shoe with topological entropy at least
$$
\lim_n\frac{1}{n(1+\varepsilon)(1+\alpha)} \log k_n =\lim_n\frac{1}{n(1+\varepsilon)(1+\alpha)}\log \frac{\exp(n(h_{\tilde{\mu}} - \alpha))}{2lDn(1+\varepsilon)(1+\alpha)},
$$
which converges to $\frac{h_{\tilde{\mu}}-\alpha}{(1+\varepsilon)(1+\alpha)}$ as $n\to\infty$. Then Theorem~\ref{t.horseshoe} follows by taking $\alpha$ and $\varepsilon$ small enough.

\begin{remark}
A similar result for star flows can be found in~\cite{LSWW}. Instead of using the shadowing lemma, they take a small neighborhood $N$ of $K_n$ and consider the Poincar\'e return map $\cP_x$ from the neighborhood of a point $x\in N$ to a neighborhood of $\phi_{m(x)+\tau_x}(x)\in N$. Then they show that for every $x,y\in K_n$, the connected component of $P_x(N)$ crosses the connected component of $P^{-1}_y(N)$, thus giving a horseshoe with $\Card K_n$ many components. 
\end{remark}
\begin{proof}[Proof of Corollary~\ref{mc.robusthexpansive}]
Recall that $\Lambda$ is a Lorenz-like class if it is a sectional hyperbolic, Lyapunov stable chain recurrent class. Then Corollary~\ref{mc.robusthexpansive} follows from Theorem~\ref{m.h-expansive}, ~\ref{m.positiveentropy}, ~\ref{m.continuity} and Lemma~\ref{l.shadowing}(f).
\end{proof}

\subsection{$C^1$ generic flows: proof of Corollary~\ref{mc.attractor}}
In this section, $\Lambda$ will be a Lorenz-like class of a $C^1$ flow $\phi_t$. Note that for every $x\in\Lambda$, the unstable set of $x$ is contained in $\Lambda$.

We need the following generic properties for $C^1$ flows. The first property is the flow version of a well known property for $C^1$ generic diffeomorphisms, which can be found in~\cite{BC}.

\begin{proposition}\label{p.homoclinicclass}
$C^1$ generically, every chain recurrent class $C$ of $\phi_t$ that contains a periodic point $p$ coincides with the homoclinic class of $\Orb(p)$. In particular, $C$ is transitive. 
\end{proposition}

The next property is a simple application of the connecting lemma in~\cite{BC}, applied to a branch of stable manifold of the singularity and the unstable manifold of the periodic orbit $\Orb(p)$.
\begin{proposition}\label{p.generic}
Let $C$ be a chain recurrent class for a $C^1$ generic flow $\phi_t$, such that $C$ contains a hyperbolic singularity $\sigma$ and a hyperbolic periodic point $p$. Assume that on $\sigma$, the stable subspace $E^{cs}_\sigma$ has a dominated splitting $E^{cs}_\sigma = E^{s}_\sigma\oplus E^c_\sigma$, where $E^c_\sigma$ is a 1-dimensional sub-bundle of $E^{cs}_\sigma$. Then the strong stable manifold $W^{s}(\sigma)$ divides the stable manifold $W^{cs}(\sigma)$ into two branches $W^{cs,+}(\sigma)$ and $W^{cs,-}(\sigma)$; furthermore, if $C\cap W^{cs,\pm}(\sigma)\setminus\{\sigma\}\ne\emptyset$, then $W^{cs,\pm}(\sigma)\cap W^u(p)\ne\emptyset$. 
\end{proposition}

\begin{proof}[Proof of Corollary~\ref{mc.attractor}]
Let $\cR$ be the residual subset of flows that are Kupka-Smale, and satisfies the above properties.  Then for every Lorenz-like class $\Lambda$, by Theorem~\ref{m.positiveentropy} and the variational principle, there is a hyperbolic measure $\mu$ with positive entropy supported on $\Lambda$. Apply Theorem~\ref{t.horseshoe}, Lemma~\ref{l.shadowing} and Proposition~\ref{p.homoclinicclass}, 
we get that $\Lambda$ is a homoclinic class of a periodic orbit $p$ and is transitive. Since $\Lambda$ is Lyapunov stable, to prove that $\Lambda$ is an attractor, it suffices to show that it is isolated, i.e., it cannot be approximated by other chain recurrent classes.

Let $\{U_n\}$ be the sequence of Lyapunov stable neighborhoods of $\Lambda$. Suppose by contradiction that $\Lambda$ is not isolated. Then one can find chain recurrent classes $\Lambda\ne C_n\subset U_n$, with $\limsup C_n \subset \Lambda$.

Every $\phi_t\in\cR$ is Kupka-Smale, thus the singularities are all hyperbolic and isolated. Thus for $n$ large, the singularities in $U_n$ are precisely those in $\Lambda$, as we have observed in (\ref{e.singularities}), Section~\ref{ss.expansive}. Since chain recurrent class is an equivalent class, we must have $C_n\cap\Lambda = \emptyset$. We can therefore assume that for all $n$, $C_n$ does not contain singularities. Taking $n$ large if necessary, we see that $C_n$ are sectional hyperbolic without singularity, thus hyperbolic. As a result, there are hyperbolic periodic point $p^n\in C_n$. Let $\Lambda_0$ be the Hausdorff limit of $\Orb(p^n)$, which is a compact, invariant and sectional hyperbolic subset of $\Lambda$.  

We claim that $\Lambda_0$ contains a singularity. If this claim is not true, then $\Lambda_0$ is hyperbolic. For $n$ large enough, the hyperbolic sets $C_n$ and $\Lambda_0$ must be homoclinic related; as a result, $C_n$ and $\Lambda$ are in fact the same homoclinic class. This contradicts our assumption that $C_n\ne\Lambda$.    

Let $\sigma\in\Lambda_0\subset\Lambda$ be a singularity.  By Lemma~\ref{l.domination.singularity}, we have a hyperbolic splitting $E^s\oplus E^c \oplus E^u$ on $T_\sigma M$ with $\dim E^c =1$, and $E^s\oplus E^c$ is the stable subspace of $T_\sigma M$. As in Lemma~\ref{l.domination.singularity} we have $W^s(\sigma)\cap \Lambda = \{\sigma\}$. By Proposition~\ref{p.generic}, $W^s(\sigma)$ divides $W^{cs}(\sigma)$ into two branches, $W^{cs,\pm}(\sigma)$.

The argument below is similar to the proof of Lemma~\ref{l.domination.singularity}.
Since $\Lambda_0$ is the Hausdorff limit of $\Orb(p^n)$, we may assume that $p^n\to\sigma$. Fix $\varepsilon>0$ small, let $t_n<0$ be the last time such that $\phi_{t_x}(p^n)\in \partial B_\varepsilon(\sigma)$. It is easy to see that $t_n\to -\infty$. Let $z^n = \phi_{t_n}(p_n)$ and $z^n\to z$, then $z\in W^{cs}(\sigma)$. 

We may assume that $z\in W^{cs,+}(\sigma )$. By Proposition~\ref{p.generic}, we can take $a\in W^{cs,+}(\sigma)\cap W^u(\Orb(p))$, where $p$ is a periodic point in $\Lambda$. Take $s$ such that $a \in W^u(\phi_s(p))$ and a disk $D$ with $a \in D \subset W^u(\phi_s(p))$. Then by the $\lambda$-lemma, $\phi_t(D)$ approximated $W^u(\sigma)$, and $\{\phi_t(x):t>0,x\in D\}$ is a submanifold that is tangent to the $F^{cu}$ bundle, with dimension $\dim F^{cu}$. Thus $W^s(z^n)\pitchfork\{\phi_t(x):t>0,x\in D\}\ne\emptyset$. 

On the other hand, $\phi_t(D)$, as a subset of $W^u(\phi_s(p))$ with $p\in \Lambda_0\subset \Lambda$, must be contained in $\Lambda$. This shows that $W^s(z^n)\cap\Lambda \ne \emptyset$. In particular, for $n$ large, we have
$$
d(\phi_t(p_n),\Lambda) \to 0 \mbox{ as } t\to\infty.
$$
As $p_n\in\C_n$, this shows that $C_n$ and $\Lambda$ are the same chain recurrent class, a contradiction.
\end{proof}

\appendix
\section{On non-hyperbolic singularities}
Here we will demonstrate how to remove the assumption on the hyperbolicity of singularities in Theorem~\ref{m.h-expansive}. The theorem that we will prove is:

\begin{main}\label{m.h-expansive.non-hyperbolic}
Let $\Lambda$ be a compact invariant set that is sectional hyperbolic for a $C^1$ flow $\phi_t$. Then there is a neighborhood $U$ of $\Lambda$, such that $\phi_t|_{U}$ is entropy expansive. 
\end{main}

Recall that Theorem~\ref{m.positiveentropy} was proven using Theorem~\ref{m.h-expansive} and Theorem~\ref{t.positiveentropy}, where the later  does not require any information on the singularity. This allows one to easily obtain the following version of Theorem~\ref{m.positiveentropy}, without assuming the hyperbolicity of singularities:

\begin{main}\label{m.positiveentropy.non-hyperbolic}
Let $\Lambda$ be a compact invariant set that is Lyapunov stable and sectional hyperbolic for a $C^1$ flow $\phi_t$. Then $h_{top}(\phi_t|_\Lambda)>0$.
\end{main}
 
In order to prove Theorem~\ref{m.h-expansive.non-hyperbolic}, we use the same argument as in Section~\ref{ss.expansive} by showing that all ergodic measures are $\varepsilon$-almost entropy expansive. The singularities being hyperbolic or not does not affect the measures that are supported away from singularities. In other words, Proposition~\ref{p.measure.away} remains valid.

To deal with measures whose support is close to some singularity, we need to establish the (topological) contracting property near singularities. This is done by a sequence of lemmas. Recall that $\tl$ is the maximal invariant set of $\phi_t$ in a small neighborhood $U$ of $\Lambda$ (when choosing $U$, there is no need to have $\Sing(\phi_t|_U)= \Sing(\phi_t|_{\Lambda})$). We refer the reader to the beginning of Section~\ref{sss.singularity} for the meaning of symbols.  

The first lemma is similar to Lemma~\ref{l.domination.singularity}. One can easily check that the proof of Lemma~\ref{l.domination.singularity} applies with slight modification.

\begin{lemma}\label{l.domination.nh.singularity}
$f|_{\Sing(X)\cap\tl}$ has a partially hyperbolic splitting $E^s\oplus E^c\oplus E^u$, where $E^c$ is a one-dimensional sub-bundle of $F^{cu}$ corresponding to an eigenvalue with norm at most one.
\end{lemma}

The next lemma describes the infinite Bowen ball at singularities. 

\begin{lemma}\label{l.bowen.singularity}~\cite{LVY}[Theorem 3.1]
For any singularity $\sigma\in\tl$, $B_\infty(\sigma,r_0/2)$ is a single point, or a $1$-dimensional center segment with length bounded by $r_0$. In the first case the singularity $\sigma$ must be isolated. In the second case, the center segment consists of singularities and saddle connections. Moreover, there are only finitely many such singularities and center segments.
\end{lemma} 
The finiteness of such singularities and center segments comes from the fact that every center segment must be contained in $B(\sigma,r_0)$, and there can only be one such segment in each $r_0$-ball. 

This lemma allows one to write $\Sing(X)\cap\tl \subset \{\sigma_1,\ldots,\sigma_m\}\cup I_1\cup\cdots\cup I_k$, where $\{\sigma_1,\ldots,\sigma_m\}$ are isolated singularities, and $\{I_1,\ldots,I_k\}$ are center segments. Each $I_i$ is fixed by $f=\phi_1$ and is contained in $B_\infty(\sigma_i,r_0/2)$ for $\sigma_i\in I_i$. We may assume that for $i\ne j$, $I_i$ and $I_j$ intersect (if they intersect at all) at boundary points, which must be a singularity. Taking a double cover if necessary, we can assume that $E^c$ is orientable, which allows us to label the end points of $I_i$ as {\em left extremal point $\sigma^-_i$}    and {\em right extremal point $\sigma^+_i$}.

Next we describe the dynamics near each center segment. Recall that $\cF^{cu}_\sigma$ are the fake foliations near $\sigma$, given by the dominated splitting $E^s\oplus F^{cu}$. We denote for every sub-center segment $I\subset I_i$,
$$
W^{uu}_\delta(I) = \bigcup_{x\in I}\cF^{cu}_{\sigma_i}(x,\delta), \mbox{ and } \cN_\delta(I) = \bigcup_{y\in W^{uu}_\delta(I)}\cF^s_{\sigma_i}(x,\delta).
$$
Then $\cN_\delta(I)$ is a ``box'' containing $I$, with size $\delta$. 

For each singularity $\sigma$, We can treat define $\cN_\delta(\sigma)$ in a similar way. In this case, $\cN_\delta(\sigma)$ is a co-dimensional one sub-manifold containing $\sigma$.

The next lemma states that for any non-trivial invariant measure $\mu$, generic point cannot  approximate the interior of each  segment $I_i$. 

\begin{lemma}\label{l.centersegment}
There are $\delta$, $r_2>0$ and finitely many center segments $I_{i,j}\subset I_i$ for $j=1,\ldots,k_i$ with $\length (I_{i,j})<r_2$ and $\Sing(X)\cap I_i = \cup_{j=1}^{k_i}\Sing(X)\cap I_{i,j}$, such that for any non-trivial invariant, ergodic measure $\mu$, we have $\mu(\cN_\delta(I_{i,j}))=0$ for every $I_{i,j}$. 
\end{lemma} 
\begin{proof}
Note that $I_i$'s are normally hyperbolic sub-manifolds. By the stable manifold theorem, for each $x\in I_i$, $\cF^s_{\sigma_i}(x,r_0)$ is the local strong stable manifold of $x$. Similar result holds for the local strong unstable manifold. Moreover,
$$
W^s_{loc}(I_i) = \bigcup_{x\in I_i}\cF^s_{\sigma_i}(x,r_0),
$$
and the same holds for $W^u_{loc}(I_i)$.

Now we fix $r_2>0$ small enough, and divide $I_i$ into $k_i$ many sub-center segments $\{I_{i,j}\}$, each with length less than $r_2$, such that the boundary points of $I_{i,j}$ are singularities. In particular, each $I_{i,j}$ contains at least two singularities. Note that during this process, we may not  have $\cup_{j=1}^{k_i}I_{i,j} = I_i$, especially if $I_i$ contains a saddle connection with length larger than $r_2$.

Suppose that there is a non-trivial ergodic measure $\mu$ and $\delta_n\to 0$ with $$\mu(N_{\delta_n}(I_{i,j}))>0.$$ Since $\mu$ is non-trivial, we must have $\mu(W^s_{loc}(I_i)) = \mu(W^u_{loc}(I_i))=0$. This allows us to take 
$$
x^n \in N_{\delta_n}(I_{i,j})\setminus (W^s_{loc}(I_i)\cup W^u_{loc}(I_i)).
$$
The negative iteration of $x^n$ must leave $ N_{\delta_n}(I_{i,j})$. Take $t_n<0$ the last time such that $\phi_{t_n}(x^n)\in \partial N_{\delta_n}(I_{i,j})$. Then $t_n\to-\infty$ as $n\to\infty$. We may suppose that 
$$
\phi_{t_n}(x^n)\to x^*\in W^s_{loc}(I_{i,j})\setminus I_{i,j}.
$$
If $r_2$ is taken small enough, then $x^*$ is close to the stable manifold of the singularity contained in $I_{i,j}$. By continuity, $X(x^*)$ and $X(\phi_{t_n}(x^n))$ are tangent to the $E^s$ cone. Now one can apply the standard argument at the end of the proof of Lemma~\ref{l.domination.singularity} to get a contradiction.
\end{proof}

For every $0<\delta_1<r_2$, if $\sigma$ is an end point of some segment $I_{i,j}$, then  $\cN_\delta(\sigma)$ cuts the ball $B_{\delta_1}(\sigma)$ into two components, which we denote by $B^+_{\delta_1}(\sigma)$ and $B^-_{\delta_1}(\sigma)$. One of the components $B^\pm_{\delta_1}(\sigma)$ will intersect with $N_{\delta}(I_{i,j})$, in which case the half ball will have zero measure for any non-trivial ergodic measure $\mu$, according to the previous lemma. 
In this case we can write the other component (which does not intersect with $N_{\delta}(I_{i,j})$) as $B^h_{\delta_1}(\sigma)$. It is possible that there is another center segment $I_{i,j'}$, which intersect $I_{i,j}$ as $\sigma$. In this case, both components of $B_{\delta_1}(\sigma)$ must have zero measure for every non-trivial ergodic measure, due to the previous lemma. If this happens, we do not need to consider any of these two components.
Otherwise, there is no singularity inside the half ball $B^h_{\delta_1}(\sigma)$, according to the construction of $I_{i,j}$.

If $\sigma$ is an isolated singularity, then $\cN_\delta(\sigma)$ cuts the ball $B_{\delta_1}(\sigma)$ into two components. Unlike the previous case, both of these two components may have positive measure for some measure $\mu$. In this case, it is convenient to treat the isolated singularity $\sigma$ as a trivial center segment, and denote the two component of $B_{\delta_1}(\sigma)$  as $B^h_{\delta_1}(\sigma^\pm)$ respectively.

To summarize, we get a finite subset $A\in\Sing(X)\cap\tl$ (with each isolated singularity appears twice in $A$) and a collection of half balls $\{B^h_{\delta_1}(\sigma)\}_{\sigma\in A}$, each of which are singularity-free (of course, other than $\sigma$ itself) and may have positive measure for some invariant measure $\mu$. Furthermore, for every ergodic measure $\mu$, typical points of $\mu$ can only approximate a singularity by going through one of these half balls. Note that for $\sigma \in A$, $\cF^c_{\sigma_i}(\sigma, r_2)$ is also cut by $\cN_\delta(\sigma)$ into two branches, one of which intersects with the half ball $B^h_{\delta_1}(\sigma)$. We will denote  $$\cF^{c,h}_{\sigma_i}(\sigma, \delta_1) = \cF^{c}_{\sigma_i}(\sigma, r_2)\cap B^h_{\delta_1}(\sigma).$$

The next lemma establishes the contracting property along $\cF^{c,h}_{\sigma_i}(\sigma, r_2)$.

\begin{lemma}\label{l.top.contracting}
There is $\delta_1>0$, such that for every $\sigma\in A$, if $\mu(B^h_{\delta_1}(\sigma))>0$ for some non-trivial invariant measure $\mu$, then $\cF^{c,h}_{\sigma_i}(\sigma, \delta_1)$ is topological contracting. More precisely, for every $x\in \cF^{c,h}_{\sigma_i}(\sigma, \delta_1)$ we must have $f^nx\to\sigma$ as $n\to+\infty$.
\end{lemma}

In other words, if a half ball $\mu(B^h_{\delta_1}(\sigma))>0$ can be `seen' by some measure $\mu$, then the center direction must be topologically contracting.
\begin{proof}
Since $\cF^{c,h}_{\sigma_i}(\sigma, \delta_1)$ is invariant and contains no singularity, it must be topological contracting or expanding. If it is topological expanding, then $\cF^{c,h}_{\sigma_i}(\sigma, \delta_1)$ belongs to the unstable set of $\sigma$. We claim that there must be  $\delta_2<\delta_1$ such that $\mu(B^h_{\delta_2}(\sigma))=0$ for every non-trivial $\mu$. Since $A$ is a finite set, the lemma follows by shrinking  $\delta_1$ a finitely number of times.   

It remains to prove this claim. Assume by contradiction that there is a sequence $\delta_n\to 0$ and a measure $\mu$, such that 
$\mu(B^h_{\delta_n}(\sigma))>0$. Similar to the proof of the previous lemma, we can take
$$
x^n\in B^h_{\delta_n}(\sigma)\setminus(W^u(\sigma)\cup \cF^s(\sigma)).
$$
Take $t_n<0$ the last time that $\phi_{t_n}(x^n)\in \partial B^h_{\delta_n}(\sigma)$. Then $t_n\to -\infty$. Since $\cF^{c,h}_{\sigma_i}(\sigma, \delta_1)$ is topological expanding, we can take $\phi_{t_n}(x^n)\to x^*\in\cF^s(\sigma)$. The same argument in the proof of Lemma~\ref{l.domination.singularity} will create a contradiction.
\end{proof}

Thus far, we have shown that:
\begin{itemize}
\item there is a partially hyperbolic splitting $E^s\oplus E^c \oplus E^u$ on the set of singularities;
\item there are finitely many singularity-free half balls  $\{B^h_{\delta_1}(\sigma)\}_{\sigma\in A}$, such that the orbit of every typical point $x$ (with respect to some non-trivial ergodic measure) can only approximate a singularity by going through  these half balls;
\item if a half ball $B^h_{\delta_1}(\sigma)$ can be ``seen'' by a non-trivial measure $\mu$, then the center direction of $\sigma$ must be topological contracting.
\end{itemize}
In view of Remark~\ref{r.fakefoliation}, Lemma~\ref{l.expansion.near} can be proven using the same argument. One only need to replace the foliation $\cF^{cs}_S$ (given by the hyperbolic splitting on $\Sing(X)$) by the fake foliation $\hat{\cF}^{cs}_S$, generated by the partially hyperbolic splitting $E^s\oplus E^c \oplus E^u$ inside a neighborhood of the singularities. Then one can define the one-dimensional center fake foliation  $\bar{\cF}^c$ as the intersection of $\hat{\cF}^{cs}_S$ and $\cF^{cu}_x$, which will give a local product structure near the neighborhood of singularities. The rest of the proof of Lemma~\ref{l.expansion.near} remains unchanged.

\section*{Acknowledgements} The authors are grateful to the anonymous referees for their careful reading and helpful comments.

\end{document}